\documentclass[12pt,a4paper,twoside,reqno]{amsart}

\usepackage{tikz}
\usepackage{slashed}
\usepackage{amsmath,amscd}
\usepackage{amssymb,amsfonts,mathrsfs}
\usepackage[driver=pdftex,margin=3cm,heightrounded=true,centering]{geometry}
\usepackage{JK}
\usepackage[colorlinks=true,linkcolor=blue,citecolor=blue]{hyperref}

\newtheorem{thm}{Theorem}[section]

\usepackage{graphicx}

\begin{document}
\title[Morita invariance of unbounded bivariant $K$-theory]{Morita invariance of unbounded bivariant $K$-theory}
\author{Jens Kaad}
\address{Department of Mathematics and Computer Science, The University of Southern Denmark, Campusvej 55, DK-5230 Odense M, Denmark}
\email{jenskaad@hotmail.com}

\thanks{The author is partially supported by the Villum Foundation (grant 7423).}


%
%
%


\begin{abstract}
We introduce a notion of Morita equivalence for non-selfadjoint operator algebras equipped with a completely isometric involution (operator $*$-algebras). We then show that the unbounded Kasparov product by a Morita equivalence bimodule induces an isomorphism between equivalence classes of twisted spectral triples over Morita equivalent operator $*$-algebras. This leads to a tentative definition of unbounded bivariant $K$-theory and we prove that this bivariant theory is related to Kasparov's bivariant $K$-theory via the Baaj-Julg bounded transform. Moreover, the unbounded Kasparov product provides a refinement of the usual interior Kasparov product. We illustrate our results by proving $C^1$-versions of well-known $C^*$-algebraic Morita equivalences in the context of hereditary subalgebras, conformal equivalences and crossed products by discrete groups.
\end{abstract}

\maketitle
\tableofcontents
\section{Introduction}
Morita equivalences for unital rings was first introduced by Kiita Morita in \cite{Mor:DMA}. Two unital rings $\sR$ and $\sS$ are Morita equivalent when there exists a full idempotent $e \in M_n(\sR)$ in the $(n \times n)$-matrices over $\sR$ such that
\[
\sS \cong e M_n(\sR) e 
\]
as unital rings. This notion of Morita equivalence may be translated to the setting of unital $*$-algebras: Two unital $*$-algebras $\sA$ and $\sB$ are Morita equivalent when there exists a full selfadjoint idempotent $p \in M_n(\sB)$ such that
\[
\sA \cong p M_n(\sB) p 
\]
as unital $*$-algebras.

Morita equivalences of unital $*$-algebras can be successfully combined with Alain Connes' notion of noncommutative geometry, \cite{Con:NCG, Con:GCM}. Indeed, two Morita equivalent unital $*$-algebras $\sA$ and $\sB$ admit the same spectral triples in the following sense: For an arbitrary spectral triple $(\sB,H,D)$ (with $\sB$ acting unitally on $H$) we obtain a new spectral triple
\[
p \sB^n \ot_{\sB} (\sB,H,D) := (\sA, p H^{\op n}, p \, \T{diag}(D) \, p )
\]
but this time over the unital $*$-algebra $\sA$, \cite{Con:GCM}. This operation (which is the most simple yet interesting version of an unbounded Kasparov product) then yields an isomorphism
\begin{equation}\label{eq:speiso}
p \sB^n \ot_{\sB} : \T{Spectral}(\sB) / \sim_{\T{bdd}} \, \, \to \T{Spectral}(\sA) / \sim_{\T{bdd}} 
\end{equation}
between the spectral triples over $\sB$ and the spectral triples over $\sA$ at least modulo the equivalence relation ``$\sim_{\T{bdd}}$'' generated by bounded perturbations and unitary equivalences.
\medskip

Morita equivalences for unital $*$-algebras was extended by Marc Rieffel to the setting of $\si$-unital $C^*$-algebras in \cite{Rie:MEC} (see also \cite{Rie:MEO,BrGrRi:SIS,Bro:SIH}). This extension relies on the theory of Hilbert $C^*$-modules over $C^*$-algebras and uses the interior tensor product of $C^*$-correspondences instead of the algebraic module tensor product. Two $\si$-unital $C^*$-algebras are Morita equivalent when there exists a full Hilbert $C^*$-module $X$ over $B$ and an isomorphism of $C^*$-algebras
\[
A \cong \sK(X)
\]
where $\sK(X) \cong X \hot_B X^*$ denotes the $C^*$-algebra of compact operators on $X$. The reader should be aware that the Hilbert $C^*$-module $X$ might not be finitely generated projective: It is only required to be countably generated over $B$ in the sense of Hilbert $C^*$-modules.

For the purposes of this paper, the main observation is that two separable Morita equivalent $C^*$-algebras $A$ and $B$ have the same analytic $K$-homology (but the same result also holds with coefficients in an arbitrary $\si$-unital $C^*$-algebra). In fact, the interior Kasparov product by the Morita equivalence bimodule $X$ induces a bijection
\begin{equation}\label{eq:anaiso}
[X] \hot_{B} : K^*(B) \to K^*(A)
\end{equation}
of abelian groups, \cite{Kas:OFE}. Remark here that $X$ defines a class $[X] \in KK_0(A,B)$ in the bounded bivariant $K$-theory of the pair of separable $C^*$-algebras $(A,B)$.
\medskip

The reader who is trained in noncommutative geometry will at this moment notice that the isomorphism in \eqref{eq:speiso} is a much more refined statement than the isomorphism in \eqref{eq:anaiso}. Indeed: The Baaj-Julg bounded transform provides a passage from spectral triples to classes in analytic $K$-homology, but this transform is subject to a substantial loss of information, \cite{BaJu:TBK}. Not only is the $*$-algebra of coordinates $\sA$ replaced by a suitable $C^*$-completion $A$ but on top of that, one looses almost all information about the eigenvalues of the unbounded selfadjoint operator $D$ (only their signs are kept). On the other hand, if two spectral triples agree up to a bounded perturbation, then the eigenvalues of the corresponding unbounded operators are bound to have the same growth properties (thus the resolvents will lie in the same Schatten ideal).
\medskip
%
%

In this paper we will introduce a notion of Morita equivalence for a certain class of dense $*$-subalgebras of a $\si$-unital $C^*$-algebra. These $*$-subalgebras are called operator $*$-algebras since they are $*$-algebraic versions of non-selfadjoint operator algebras, see \cite{BlMe:OAM}. In particular, it turns out that operator $*$-algebras can be characterized, up to completely bounded isomorphisms, as being exactly those closed involutive subalgebras of bounded operators on a Hilbert space where the involution agrees with the Hilbert space adjoint operation up to conjugation by a symmetry, \cite{BlKaMe:OAG}. There is thus an abundance of operator $*$-algebras available from an abstract point of view, but even more importantly there is a minimal operator $*$-algebra $\C A$ associated to any spectral triple $(\sA,H,D)$. Indeed, as a Banach $*$-algebra $\C A$ is simply the minimal domain of the derivation $[D, \cd] : \sA \to \sL(H)$. 
\medskip

The relevant bimodules over operator $*$-algebras are called operator $*$-correspondences. These bimodules come equipped with an inner product taking values in an operator $*$-algebra and they appear as a natural analogue of $C^*$-correspondences in the world of operator bimodules, see \cite{BlMe:OAM}. In particular, it can be proved that operator $*$-correspondences are, up to completely bounded isomorphism, exactly those closed inner product bimodules of bounded operators on a Hilbert space where the inner product agrees with the pairing $(T,S) \mapsto T^* S$ up to conjugation of the Hilbert space adjoint $T^*$ by a symmetry, \cite{BlKaMe:OAG}. Hence, when the symmetry has no negative eigenvalues (is equal to the identity) one obtains the usual class of $C^*$-correspondences over a pair of $C^*$-algebras. To illustrate the richness and interest of this class of bimodules from the point of view of noncommutative geometry it should suffice to notice that operator $*$-correspondences arise naturally as minimal domains of hermitian connections, \cite{BlKaMe:OAG}.
\medskip

The main theorem of this text can now be stated as follows:

\begin{thm}
Suppose that $\C A$ and $\C B$ are Morita equivalent operator $*$-algebras, then the unbounded Kasparov product with the Morita equivalence bimodule $\C X$ induces an isomorphism:
\[
\C X \hot_{\C B} : \T{Mod-Spec}(\C B)/ \sim_{\T{bmp}} \, \to \T{Mod-Spec}(\C A)/\sim_{\T{bmp}}
\]
between the modular spectral triples over $\C B$ and the modular spectral triples over $\C A$ modulo the equivalence relation ``$\sim_{\T{bmp}}$'' generated by bounded modular perturbations and unitary equivalences. A similar result holds with coefficients in an auxiliary $\si$-unital $C^*$-algebra $C$.
\end{thm}

To exemplify our main theorem we analyze the following situations:
\begin{enumerate}
\item Hereditary subalgebras of operator $*$-algebras.
\item Conformal equivalences (of Riemannian metrics).
\item Discrete groups acting freely and properly on Riemannian manifolds.
\end{enumerate}
More precisely, we establish operator $*$-algebra (or $C^1$) versions of a number of Morita equivalence results that were previously only known to hold for $C^*$-algebras. As a consequence we obtain the corresponding isomorphism between the (equivalence classes of) modular spectral triples over the various operator $*$-algebras involved. This is a substantially more detailed result than having an isomorphism of analytic $K$-homology groups. Remark also that the second entry is a triviality for $C^*$-algebras since the Riemannian metric can not be detected as the level of continuous functions.
\medskip

The modular spectral triples over an operator $*$-algebra appear as a twisted analogue of spectral triples. This means that the usual commutator condition is replaced by a twisted commutator condition where the twist is expressed in terms of an extra modular operator $\Ga : H \to H$. The reasons for including this extra twist are deep and related to the non-existence of frames in a general operator $*$-correspondence, \cite{Kaa:DAH, Kaa:UKM}. To understand what this means (at least to some extent), it suffices to remember that there are many (non-compact) Riemannian manifolds that do not admit a partition of unity with uniformly bounded exterior derivatives. This observation has severe consequences for the selfadjointness properties of a symmetric first order differential operator $D$ acting on a smooth hermitian vector bundle over such a Riemannian manifold. The presence of the modular operator $\Ga$ in this setting then means that we are really computing with respect to a Riemannian metric that is conformally equivalent to the original one and the conformal factor is given by the operator $\Ga^{-2}$. More precisely, we are not studying $D$ directly but rather the operator $\Ga^{1/2} D \Ga^{1/2}$ which might very well have better selfadjointness properties.

In a noncommutative setting, the modular group of automorphisms associated with the modular operator $\Ga$ (thus the action of the real line on $\sL(H)$ given by $T \mapsto \Ga^{it} T \Ga^{-it}$, $t \in \rr$) starts to play a substantial role. In particular, the non-triviality of this group of automorphisms (on elements coming from the algebra of coordinates) is responsible for the twisting of the commutator condition in the definition of a modular spectral triple. The notion of a modular spectral triple appearing in this text (and in \cite{Kaa:UKM}) is thus tightly related to (but different from) the notion of a twisted spectral triple as introduced by Alain Connes and Henri Moscovici in \cite{CoMo:TST}.
\medskip

The idea of carrying out the interior Kasparov product directly at the level of unbounded cycles dates back to Alain Connes, \cite{Con:GCM}. The recent interest in this approach is however due to Bram Mesland's PhD-thesis where this idea is extended far beyond the context of finitely generated projective modules, \cite{Mes:UCN}. Moreover, this thesis is also the first place where the theory of operator spaces and notably operator modules is applied to describe the domain of a hermitian connection from an analytic point of view. Since then, a series of papers has developed these ideas in various directions and notably so by providing more advanced constructions of unbounded Kasparov products, \cite{KaLe:SFU, MeRe:MCU, BrMeSu:GSU, Kaa:UKM}. 

The unbounded Kasparov product by an operator $*$-correspondence $\C X$ satisfying an extra compactness condition was constructed and investigated in details in \cite{Kaa:UKM}, and this is the version of the unbounded Kasparov product that we apply in this text. We show that this kind of unbounded Kasparov product yields a well-defined operation (independent of various choices):
\[
\C X \hot_{\C B} : \T{Mod-Spec}(\C B)/ \sim_{\T{bmp}} \to \T{Mod-Spec}(\C A)/ \sim_{\T{bmp}}
\]
Moreover, we establish that the Baaj-Julg bounded transform yields a well-defined map
\[
F : \T{Mod-Spec}(\C B)/ \sim_{\T{bmp}} \to K^*(B) \q D \mapsto D(1 + D^2)^{-1/2}
\]
with values in the analytic $K$-homology of the $C^*$-completion of $\C B$. The relation between the unbounded Kasparov product and the interior Kasparov product can then be explained by the commutative diagram:
\[
\begin{CD}
\T{Mod-Spec}(\C B)/ \sim_{\T{bmp}} \, @>{\C X \hot_{\C B}}>> \T{Mod-Spec}(\C A)/ \sim_{\T{bmp}} \\
@V{F}VV @V{F}VV \\
K^*(B) @>{X \hot_B }>> K^*(A)
\end{CD}
\]
where $X$ denotes the $C^*$-completion of $\C X$. All of these results remain true with coefficients in a $\si$-unital $C^*$-algebra $C$.
\medskip

Morita equivalences for operator algebras were introduced by David Blecher, Paul Muhly and Vern Paulsen in \cite{BlMuPa:COM} (see also \cite{Ble:MTA,BlMuNa:MEO}). It would be possible to translate their notion of Morita equivalence directly to the context of operator $*$-correspondences and operator $*$-algebras using a version of the Haagerup tensor product for operator $*$-correspondences (and keeping track of inner products and involutions). This would however result in a stronger notion of Morita equivalence than the one we apply in this text. The operator $*$-correspondences, we are concerned with here, all sit densely inside a $C^*$-correspondence and we deem two such operator $*$-correspondences $\C X, \C Y \su Z$ to be equivalent when the inner product $\inn{\cd,\cd}_Z : \C X \ti \C Y \to B$ induces a completely bounded pairing $\inn{\cd,\cd}_Z : \C X \ti \C Y \to \C B$. It is therefore also unlikely that the operator $*$-algebras considered in the examples section of this paper would be Morita equivalent in the sense of Blecher, Muhly and Paulsen (even if the involutions were disregarded). 
\medskip

The present paper is structured as follows: We start in Section \ref{s:opealg}, Section \ref{s:opestar} and Section \ref{s:comple} by reviewing the basic theory of operator $*$-algebras and operator $*$-correspondences together with their associated $C^*$-completions. In Section \ref{s:equrel} we introduce an equivalence relation on operator $*$-correspondences. In Section \ref{s:algope} we study direct sums and tensor products of operator $*$-correspondences and relate these constructions to the direct sum and interior tensor product of the associated $C^*$-completions. In Section \ref{s:morita} we introduce our notion of Morita equivalence for operator $*$-algebras. This concept is formulated in categorical terms: We construct a category where the objects are operator $*$-algebras and the morphisms are equivalence classes of operator $*$-correspondences. In Section \ref{s:unbmod} we recall the notion of an unbounded modular cycle and in Section \ref{s:equrelII} we introduce the equivalence relation given by bounded modular perturbations of unbounded modular cycles. The corresponding quotient space provides a tentative definition of unbounded bivariant $K$-theory. In Section \ref{s:unbkas} we show that the unbounded Kasparov product by an operator $*$-correspondence (satisfying an extra compactness condition) descends to a well-defined operation on unbounded bivariant $K$-theory. In Section \ref{s:baajul} we show that the Baaj-Julg bounded transform is compatible with bounded modular perturbations of unbounded modular cycles and hence that we obtain a homomorphism from unbounded bivariant $K$-theory with values in the bounded bivariant $K$-theory of the $C^*$-completions of our operator $*$-algebras. In Section \ref{s:compro} we prove that the unbounded Kasparov product agrees with the interior Kasparov product after taking bounded transforms. In Section \ref{s:geomor} we provide our geometric examples of Morita equivalences and hence of isomorphisms of unbounded bivariant $K$-theories.

\section{Operator $*$-algebras}\label{s:opealg}
For a vector space $V$ over the complex numbers $\cc$ we let $M_m(V)$, $m \in \nn$, denote the vector space of $(m \ti m)$-matrices over $V$. When $V$ is a closed subspace of the bounded operators $\sL(H)$ on a Hilbert space $H$ we may equip $M_m(V)$ with the operator norm inherited from $M_m\big( \sL(H) \big) \cong \sL(H^{\op m})$. The abstract properties of this sequence of Banach spaces are crystallized in the following:


\begin{dfn}\label{d:opspa}
A vector space $\C X$ over $\cc$ is an \emph{operator space} when it is equipped with a norm $\| \cd \|_{\C X} : M_m(\C X) \to [0,\infty)$ for all $m \in \nn$ such that the following holds:
\begin{enumerate}
\item $\C X \cong M_1(\C X)$ is complete in the norm $\| \cd \|_{\C X} : M_1(\C X) \to [0,\infty)$;
\item For each $m \in \nn$, $\xi \in M_m(\C X)$ and $\la,\mu \in M_m(\cc)$ we have the inequality
\[
\| \la \cd \xi \cd \mu \|_{\C X} \leq \| \la \|_{\cc} \cd \| \xi \|_{\C X} \cd \| \mu \|_{\cc}
\]
where $(\la \cd \xi \cd \mu)_{ij} = \sum_{k,l = 1}^m \la_{ik} \cd \xi_{kl} \cd \mu_{lj}$ for all $i,j \in \{1,\ldots,m\}$ and where the norm on $M_m(\cc)$ is the unique $C^*$-algebra norm.
\item For each $\xi \in M_m(\C X)$ and $\eta \in M_k(\C X)$ we have that
\[
\| \xi \op \eta \|_{\C X} = \max\{ \| \xi \|_{\C X} , \| \eta \|_{\C X} \} 
\]
where $\xi \op \eta \in M_{m + k}(\C X)$ refers to the direct sum of matrices.
\end{enumerate}
A linear map $\phi : \C X \to \C Y$ between two operator spaces is \emph{completely bounded} when there exists a constant $C > 0$ such that
\[
\| \phi(\xi) \|_{\C Y} \leq C \cd \| \xi \|_{\C X}
\]
for all $\xi \in M_m(\C X)$ and all $m \in \nn$. The \emph{completely bounded norm} of such a completely bounded map is defined by 
\[
\| \phi \|_{\T{cb}} := \inf\big\{ C \in [0,\infty) \mid \| \phi(\xi) \| \leq C \cd \| \xi \| \, , \, \, \forall \xi \in M_m(\C X) \, , \, \, m \in \nn \big\}
\]
\end{dfn}

Indeed, by a theorem of Ruan, for any abstract operator space $\C X$ there exist a closed subspace $V \su \sL(H)$ of the bounded operators on some Hilbert space $H$ together with a completely isometric isomorphism $\phi : \C X \to V$, see \cite[Theorem 3.1]{Rua:SCA}.
\medskip

When the Banach spaces $M_m(\C X)$ associated to an operator space $\C X$ are in fact Banach algebras, where the multiplication on the higher matrix algebras arises from the multiplication on $\C X$ through matrix multiplication we refer to $\C X$ as an ``operator algebra'':

\begin{dfn}
An operator space $\C A$ is an \emph{operator algebra} when it comes equipped with a multiplication $m : \C A \ti \C A \to \C A$ such that
\begin{enumerate}
\item $\C A$ becomes a Banach algebra over $\cc$;
\item We have the inequality
\[
\| x \cd y \|_{\C A} \leq \| x \|_{\C A} \cd \|y\|_{\C A} \q \T{for all } m \in \nn \T{ and } x,y \in M_m(\C A)
\]
where $(x \cd y)_{ij} = \sum_{k = 1}^m x_{ik} \cd y_{kj}$ for all $i,j \in \{1,\ldots,m\}$.
\end{enumerate}
\end{dfn}

By a theorem of Blecher, any operator algebra $\C A$ is completely isomorphic to a concrete operator algebra. Thus, there exist a closed subalgebra $\C B \su \sL(H)$ of the bounded operators on some Hilbert space $H$ together with a completely bounded algebra isomorphism $\phi : \C A \to \C B$ with the additional property that the inverse $\phi^{-1} : \C B \to \C A$ is completely bounded as well, see \cite[Theorem 2.2]{Ble:CBC}.
\medskip

Finally, when the sequence of Banach algebras $M_m(\C A)$, $m \in \nn$, arising from an operator algebra $\C A$ really consists of Banach $*$-algebras where the ``higher'' involutions are compatible with the involution on $\C A$ we refer to $\C A$ as an ``operator $*$-algebra'':

\begin{dfn}
An operator algebra $\C A$ is an \emph{operator $*$-algebra} when it comes equipped with an involution $* : \C A \to \C A$ such that
\begin{enumerate}
\item $\C A$ becomes a Banach $*$-algebra;
\item We have the identity
\[
\| x^* \|_{\C A} =  \| x \|_{\C A} \q \T{for all } m \in \nn \T{ and } x \in M_m(\C A)
\]
where $(x^*)_{ij} = (x_{ji})^*$ for all $i,j \in \{1,\ldots,m\}$.
\end{enumerate}
\end{dfn}

Any operator $*$-algebra $\C A$ is completely isomorphic to a concrete operator $*$-algebra. More precisely, there exist a closed subalgebra $\C B \su \sL(H)$, a selfadjoint unitary operator $U : H \to H$ and a completely bounded algebra isomorphism $\phi : \C A \to \C B$ (with completely bounded inverse) such that
\[
U \phi(a^*) U = \phi(a)^*
\]
where $* : \sL(H) \to \sL(H)$ denotes the adjoint operation coming from the inner product on $H$, see \cite{BlKaMe:OAG}.
\medskip

Recall that any $C^*$-algebra $A$ can be given the structure of an operator $*$-algebra. For $m \in \nn$ the matrix norm $\| \cd \|_A : M_m(A) \to [0,\infty)$ is the unique $C^*$-norm on $M_m(A)$. 

\subsection{The operator $*$-algebra of compacts}\label{ss:opecom}
We now introduce a stabilization procedure for an operator $*$-algebra $\C A$. The corresponding construction for operator algebras is standard and can be found in \cite[Section 2.2.3]{BlMe:OAM}.

Let $M(\C A)$ denote the $*$-algebra of infinite matrices with only finitely many non-zero entries in $\C A$. We will often write elements in $M(\C A)$ as infinite sums $\sum_{i,j = 1}^\infty a_{ij} e_{ij}$, $a_{ij} \in \C A$ and $a_{ij} \neq 0$ only for finitely many $i,j \in \nn$. The sum $\sum_{i, j = 1}^\infty a_{ij} e_{ij}$ is then identified with the infinite matrix with $a_{ij}$ in position $(i,j)$. 
%

The matrix norms $\| \cd \|_{\C A} : M_m(\C A) \to [0,\infty)$ then induce a norm $\| \cd \|_{\C A} : M(\C A) \to [0,\infty)$. We let $K_{\C A}$ denote the Banach $*$-algebra obtained as the completion of $M(\C A)$ in the norm $\| \cd \|_{\C A}$.

For each $n \in \nn$, we let $\pi_n : K_{\C A} \to M_n(\C A)$ denote the bounded operator induced by 
\[
\pi_n : M(\C A) \to M_n(\C A) \q \sum_{i,j = 1}^\infty a_{ij} \cd e_{ij} \mapsto \sum_{i,j = 1}^n a_{ij} \cd e_{ij}
\]

For each $m \in \nn$, we define the matrix norm
\begin{equation}\label{eq:commat}
\| \cd \|_{K_{\C A}} : M_m( K_{\C A} ) \to [0,\infty) \q \| x \|_{K_{\C A}} := \sup_{n \in \nn} \| \pi_n(x) \|_{\C A}
\end{equation}
where $\pi_n : M_m( K_{\C A} ) \to M_m\big( M_n(\C A) \big) \cong M_{m \cd n}(\C A)$ is given by applying $\pi_n : K_{\C A} \to M_n(\C A)$ entry-wise (and the last identification is given by forgetting the subdivisions).

\begin{dfn}
By the \emph{compacts over $\C A$} we will understand the operator $*$-algebra $K_{\C A}$ equipped with the $*$-algebra structure coming from $M(\C A)$ and the matrix norms $\| \cd \|_{K_{\C A}} : M_m( K_{\C A} ) \to [0,\infty)$, $m \in \nn$, defined in \eqref{eq:commat}.
\end{dfn}

The terminology ``compacts over $\C A$'' is chosen by analogy with the Hilbert $C^*$-module situation. Indeed, when $\C A$ happens to be a $C^*$-algebra (with the canonical $C^*$-norms on $M_m(\C A)$, $m \in \nn$, as matrix norms) our construction recovers the $C^*$-algebra of compact operators on the standard module $\ell^2(\nn,\C A)$ over $\C A$, see \cite[Lemma 4]{Kas:HSV}.


\section{Operator $*$-correspondences}\label{s:opestar}
We recall the definition of the relevant class of bimodules over operator algebras:

\begin{dfn}\label{d:opebim}
Let $\C A$ and $\C B$ be operator algebras. An operator space $\C X$ is an operator $\C A$-$\C B$-bimodule when the following holds:
\begin{enumerate}
\item $\C X$ is an $\C A$-$\C B$-bimodule;
\item We have the inequality
\[
\| x \cd \xi \|_{\C X} \leq  \| x \|_{\C A} \cd \| \xi \|_{\C X} \q \T{for all } m \in \nn \, , \, \, x \in M_m(\C A) \T{ and } \xi \in M_m(\C X)
\]
where $(x \cd \xi)_{ij} = \sum_{k = 1}^m x_{ik} \cd \xi_{kj}$ for all $i,j \in \{1,\ldots,m\}$;
\item We have the inequality
\[
\| \xi \cd y\|_{\C X} \leq \| \xi \|_{\C X} \cd \| y \|_{\C B} \q \T{for all } m \in \nn \, , \, \, \xi \in M_m(\C X) \T{ and } y \in M_m(\C B)
\]
where $(\xi \cd y)_{ij} = \sum_{k = 1}^m \xi_{ik} \cd y_{kj}$ for all $i,j \in \{1,\ldots,m\}$.
\end{enumerate}
\end{dfn}

Any operator $\C A$-$\C B$-bimodule $\C X$ is completely bounded isomorphic to a concrete operator bimodule in the following way: There exist closed subalgebras $\C C$ and $\C D \su \sL(H)$ and a closed subspace $\C Y \su \sL(H)$ for some Hilbert space $H$ together with completely bounded algebra isomorphisms $\phi_{\C A} : \C A \to \C C$, $\phi_{\C B} : \C B \to \C D$ and a completely bounded isomorphism $\phi_{\C X} : \C X \to \C Y$ all with completely bounded inverses such that
\begin{equation}\label{eq:modulrel}
\phi_{\C A}(x) \cd \phi_{\C X}(\xi) = \phi_{\C X}(x \cd \xi) \, \, \, \T{and} \, \, \, \,
\phi_{\C X}(\xi) \cd \phi_{\C B}(y) = \phi_{\C X}(\xi \cd y)
\end{equation}
for all $x \in \C A$, $\xi \in \C X$ and $y \in \C B$, see \cite[Theorem 2.2]{Ble:GHM}. 
\medskip

Let $\C A$ and $\C B$ be operator $*$-algebras. Any operator $\C A$-$\C B$-bimodule $\C X$ admits a formal dual operator $\C B$-$\C A$-bimodule $\C X^*$. As a vector space we have that $\C X^* := \{ \xi^* \mid \xi \in \C X\}$ agrees with the formal dual of the vector space $\C X$. The bimodule structure on $\C X^*$ is defined by $b \cd \xi^* := (\xi \cd b^*)^*$ and $\xi^* \cd a := (a^* \cd \xi)^*$ for all $\xi \in \C X$, $b \in \C B$ and $a \in \C A$. The matrix norms $\| \cd \|_{\C X^*} : M_m(\C X^*) \to [0,\infty)$, $m \in \nn$, are given by
\[
\| \xi^* \| := \| \xi \| \q \T{ for all } \xi \in M_m(\C X) 
\]
where $(\xi^*)_{ij} := (\xi_{ji})^*$ for all $i,j \in \{1,\ldots,m\}$.
\medskip

When $\C A$ and $\C B$ are operator $*$-algebras it becomes interesting to consider operator $\C A$-$\C B$-bimodules which comes equipped with an inner product taking values in the operator $*$-algebra $\C B$ (the one that acts from the right).

\begin{dfn}\label{d:herope}
Let $\C A$ and $\C B$ be operator $*$-algebras. An operator $\C A$-$\C B$-bimodule $\C X$ is an \emph{operator $*$-correspondence} from $\C A$ to $\C B$ when it comes equipped with a pairing
\[
\inn{\cd,\cd} : \C X \ti \C X \to \C B 
\]
satisfying the conditions:
\begin{enumerate}
\item $\inn{\xi, \eta \cd (b + \la)} = \inn{\xi, \eta} \cd (b + \la)$ for all $\xi,\eta \in \C X$, $b \in \C B$ and $\la \in \cc$;
\item $\inn{\xi, \eta + \ze} = \inn{\xi, \eta} + \inn{\xi, \ze}$ for all $\xi,\eta,\ze \in \C X$;
\item $\inn{\xi,\eta} = \inn{\eta,\xi}^*$ for all $\xi, \eta \in \C X$;
\item $\inn{\xi, a \cd \eta} = \inn{a^* \cd \xi, \eta}$ for all $\xi, \eta \in \C X$ and $a \in \C A$;
\item We have the inequality
\[
\| \inn{\xi,\eta} \|_{\C B} \leq \| \xi \|_{\C X} \cd \| \eta \|_{\C X} \q \T{for all } m \in \nn \T{ and } \xi, \eta \in M_m(\C X)
\]
where $( \inn{\xi,\eta} )_{ij} := \sum_{k = 1}^m \inn{\xi_{ki}, \eta_{kj}}$ for all $i,j \in \{1,\ldots,m\}$.
\end{enumerate}
We refer to the pairing $\inn{\cd,\cd}$ as the \emph{inner product}. Condition $(5)$ will be referred to as the \emph{Cauchy-Schwarz inequality}. We say that the inner product is \emph{non-degenerate} when $\big( \inn{\xi,\eta} = 0 \, \, \, \, \forall \eta \in \C X \big) \rar ( \xi = 0 ) $.
\end{dfn}

Any operator $*$-correspondence $\C X$ from $\C A$ to $\C B$ is completely bounded isomorphic to a concrete operator $*$-correspondence provided that the inner product is non-degenerate: As in the case of operator bimodules there exist completely bounded algebra homomorphisms $\phi_{\C A} : \C A \to \C C$ and $\phi_{\C B} : \C B \to \C D$ and a completely bounded isomorphism $\phi_{\C X} : \C X \to \C Y$ (all with completely bounded inverses) satisfying the relations in \eqref{eq:modulrel}. Moreover, there exists a selfadjoint unitary $U : H \to H$ such that
\[
U\phi_{\C A}(x^*)U = \phi_{\C A}(x)^* \, , \, \, U \phi_{\C B}(y^*) U = \phi_{\C B}(y)^* \, \, \, \T{and} \, \, \, \,
U \phi_{\C X}(\xi)^* U \phi_{\C X}(\eta) = \phi_{\C B}( \inn{\xi,\eta})
\]
for all $x \in \C A$, $y \in \C B$ and $\xi,\eta \in \C X$, see \cite{BlKaMe:OAG}.
\medskip

Recall that a $C^*$-correspondence from a $C^*$-algebra $A$ to a $C^*$-algebra $B$ is a (right) Hilbert $C^*$-module $X$ over $B$ together with a $*$-homomorphism $\pi : A \to \sL(X)$ where $\sL(X)$ denotes the $C^*$-algebra of bounded adjointable operators on $X$. Any $C^*$-correspondence $X$ from $A$ to $B$ can be given the structure of an operator $*$-correspondence from $A$ to $B$. The matrix norms $\| \cd \|_X : M_m(X) \to [0,\infty)$, $m \in \nn$, are defined by
\[
\| \xi \|_X := \| \inn{\xi,\xi}_X \|_B^{1/2} \q \xi \in M_m(X)
\]
Notice that $\inn{\xi,\xi}_X \in M_m(B)$ for $\xi \in M_m(X)$ and that $\| \cd \|_B : M_m(B) \to [0,\infty)$ refers to the unique $C^*$-norm, see \cite[Section 3]{Ble:AHM}.



\subsection{Row and column correspondences}\label{ss:rowcol}
Let $\C A$ and $\C B$ be operator $*$-algebras and let $\C X$ be an operator $*$-correspondence from $\C A$ to $\C B$. We are now going to use the matrix norms to construct various stabilizations of our operator $*$-correspondence. The constructions are again standard in the case of operator bimodules, see \cite[Proposition 3.1.14]{BlMe:OAM}. 

We let $M(\C X)$ denote the vector space of infinite matrices with only finitely many non-zero entries in $\C X$. We equip $M(\C X)$ with the $M(\C A)$-$M(\C B)$-bimodule structure defined by
\[
(a \cd \xi)_{ij} := \sum_{k = 1}^\infty a_{ik} \cd \xi_{kj} \q \T{and} \q (\xi \cd b)_{ij} := \sum_{k = 1}^\infty \xi_{ik} \cd b_{kj} \q i,j \in \nn
\]
Furthermore, we have the pairing $\inn{\cd,\cd} : M(\C X) \ti M(\C X) \to M(\C B)$ defined by
\begin{equation}\label{eq:matpair}
\inn{\xi,\eta}_{ij} := \sum_{k = 1}^\infty \inn{\xi_{ki}, \eta_{kj}}_{\C X} \q i,j \in \nn
\end{equation}
We remark that this pairing is compatible with the bimodule structure in the sense that the conditions $(1)$-$(4)$ of Definition \ref{d:herope} are satisfied.

The matrix norms $\| \cd \|_{\C X} : M_m(\C X) \to [0,\infty)$, $m \in \nn$, induce a norm $\| \cd \|_{\C X} : M(\C X) \to [0,\infty)$ and we let $K_{\C X}$ denote the Banach space obtained as the corresponding completion of $M(\C X)$.

As in Subsection \ref{ss:opecom} we have a bounded operator $\pi_n : K_{\C X} \to M_n(\C X)$ for all $n \in \nn$, and for each $m \in \nn$ we define the matrix norm
\[
\| \cd \|_{K_{\C X}} : M_m( K_{\C X}) \to [0,\infty) \q \| \xi \|_{K_{\C X}} := \sup_{n \in \nn} \| \pi_n(\xi) \|_{\C X}
\]
where $\pi_n : M_m(K_{\C X}) \to M_m(M_n(\C X)) \cong M_{m\cd n}(\C X)$ is obtained by applying $\pi_n : K_{\C X} \to M_n(\C X)$ entry-wise. With these definitions it may be verified that the Cauchy-Schwarz inequality is satisfied as well and we thus have the following:

\begin{dfn}
By the \emph{compacts over $\C X$} we understand the operator $*$-correspondence $K_{\C X}$ from $K_{\C A}$ to $K_{\C B}$ with bimodule structure induced by the $M(\C A)$-$M(\C B)$-bimodule structure on $M(\C X)$ and with inner product induced by the pairing $\inn{\cd,\cd} : M(\C X) \ti M(\C X) \to M(\C B)$ from \eqref{eq:matpair}.
\end{dfn}

We let $C_c(\nn,\C X) \su M(\C X)$ denote the vector subspace defined by
\[
\xi \in C_c(\nn,\C X) \lrar \Big( \xi \in M(\C X) \T{ and } \xi_{ij} = 0 \, \, \forall j \geq 2 \Big)
\]
We equip $C_c(\nn,\C X)$ with the $M(\C A)$-$\C B$-bimodule structure defined by
\[
(a \cd \xi)_{i1} := \sum_{k = 1}^\infty a_{ik} \cd \xi_{k1} \q \T{and} \q (\xi \cd b)_{i1} := \xi_{i1} \cd b
\]
for all $a \in M(\C A)$ and $b \in \C B$. We equip the bimodule $C_c(\nn,\C X)$ with the pairing $\inn{\cd,\cd} : C_c(\nn,\C X) \ti C_c(\nn,\C X) \to \C B$ defined by
\[
\inn{\xi,\eta} := \sum_{i = 1}^\infty \inn{\xi_{i1}, \eta_{i1}}_{\C X} \q \xi, \eta \in C_c(\nn,\C X)
\]

\begin{dfn}
By the \emph{column correspondence over $\C X$} we understand the operator $*$-correspondence $\ell^2(\nn,\C X)$ from $K_{\C A}$ to $\C B$ obtained as the completion of $C_c(\nn,\C X) \su K_{\C X}$. The bimodule structure is induced by the $M(\C A)$-$\C B$ bimodule structure on $C_c(\nn,\C X)$ and the inner product is induced by the pairing $\inn{\cd,\cd} : C_c(\nn,\C X) \ti C_c(\nn,\C X) \to \C B$.
\end{dfn}

An important example of a column correspondence arises when the operator $*$-algebra $\C B$ is considered as an operator $*$-correspondence from $\C B$ to $\C B$. Thus, when the bimodule structure comes from the algebra structure on $\C B$ and when the inner product $\inn{\cd,\cd} : \C B \ti \C B \to \C B$ is defined by $\inn{b_0,b_1} := b_0^* \cd b_1$. The matrix norms on this operator $*$-correspondence agrees with the matrix norms on $\C B$ considered as an operator $*$-algebra. In this case, we obtain the ``standard module over $\C B$'', $\ell^2(\nn,\C B)$, which is an operator $*$-correspondence from $K_{\C B}$ to $\C B$, see also \cite[Definition 3.3]{KaLe:SFU}.  When $\C B$ happens to be a $C^*$-algebra this construction recovers the usual standard module over $\C B$ (or the ``Hilbert space over $\C B$''), see \cite[Definition 2]{Kas:HSV}.
\medskip

We let $C_c(\nn,\C X)^t \su M(\C X)$ denote the vector subspace defined by
\[
\xi \in C_c(\nn,\C X)^t \lrar \Big( \xi \in M(\C X) \T{ and } \xi_{ij} = 0 \, \, \forall i \geq 2 \Big)
\]
We equip $C_c(\nn,\C X)^t$ with the $\C A$-$M(\C B)$-bimodule structure defined by
\[
(a \cd \xi)_{1j} := a \cd \xi_{1j} \q \T{and} \q (\xi \cd b)_{1j} := \sum_{k = 1}^\infty \xi_{1k} \cd b_{kj}
\]
for all $a \in \C A$ and $b \in M(\C B)$. We equip the bimodule $C_c(\nn,\C X)^t$ with the pairing $\inn{\cd,\cd} : C_c(\nn,\C X)^t \ti C_c(\nn,\C X)^t \to M(\C B)$ defined by
\begin{equation}\label{eq:rowpair}
\inn{\xi,\eta}_{ij} := \inn{\xi_{1i}, \eta_{1j}}_{\C X} \q \xi, \eta \in C_c(\nn,\C X)^t
\end{equation}

\begin{dfn}
By the \emph{row correspondence over $\C X$} we understand the operator $*$-correspondence $\ell^2(\nn,\C X)^t$ from $\C A$ to $K_{\C B}$ obtained as the completion of $C_c(\nn,\C X)^t \su K_{\C X}$. The bimodule structure is induced by the $\C A$-$M(\C B)$-bimodule structure on $C_c(\nn,\C X)^t$ and the inner product is induced by the pairing $\inn{\cd,\cd} : C_c(\nn,\C X)^t \ti C_c(\nn,\C X)^t \to M(\C B)$.
\end{dfn}


\section{$C^*$-completions}\label{s:comple}
We are in this text mainly interested in operator $*$-algebras sitting densely inside $C^*$-algebras and we will thus make the following:


\begin{assu}\label{a:opealg}
From now on, any operator $*$-algebra $\C A$ (with operator $*$-algebra norm $\| \cd \|_{\C A}$) will be assumed to come equipped with an additional fixed $C^*$-norm $\| \cd \|_A : \C A \to [0,\infty)$ and the associated $C^*$-completion will be denoted by $A$. The inclusion $\io : \C A \to A$ is required to be completely bounded and the $C^*$-completion is required to be $\si$-unital (thus $A$ has a countable approximate identity).
\end{assu}

To distinguish more clearly between the $C^*$-norm and the operator $*$-algebra norm we will sometimes use the notation
\[
\| \cd \|_\infty := \| \cd \|_A \q \T{and} \q \| \cd \|_1 := \| \cd \|_{\C A}
\]

In line with the above definition, our main focus will lie on operator $*$-correspondences admitting a suitable $C^*$-completion. This leads to the following:

\begin{assu}\label{a:opecor}
From now on, any operator $*$-correspondence $\C X$ from $\C A$ to $\C B$ will be assumed to satisfy the following conditions:
\begin{enumerate}
\item For all $\xi \in \C X$ we have that $\io(\inn{\xi,\xi}) \geq 0$ (where $\io : \C B \to B$ is the inclusion);
\item The implication $\big( \inn{\xi,\xi} = 0 \big) \rar \big( \xi = 0 \big)$ holds for all $\xi \in \C X$;
\item For all $a \in \C A$ and $\xi \in \C X$ we have that
\[
\| \inn{a \cd \xi, a \cd \xi} \|_\infty \leq \| a \|_\infty^2 \cd \| \inn{\xi,\xi}\|_\infty
\]
\end{enumerate}
\end{assu}

We emphasize that condition $(3)$ in the above assumption is \emph{not} automatic.

\begin{dfn}
The \emph{$C^*$-completion} $X$ of an operator $*$-correspondence $\C X$ is the completion of $\C X$ with respect to the norm $\| \cd \|_\infty : \C X \to [0,\infty)$ defined by $\| \xi \|_\infty := \| \inn{\xi,\xi} \|_\infty^{1/2}$ for all $\xi \in \C X$. 
\end{dfn}

We will sometimes denote the operator $*$-correspondence norm on $\C X$ by $\| \cd \|_1 := \| \cd \|_{\C X} : \C X \to [0,\infty)$.

\begin{lemma}
Let $\C X$ be an operator $*$-correspondence from $\C A$ to $\C B$. Then the operator $*$-correspondence structure on $\C X$ induces a $C^*$-correspondence structure (from $A$ to $B$) on the $C^*$-completion $X$ of $\C X$ and the inclusion $\io : \C X \to X$ is completely bounded.
\end{lemma}
\begin{proof}
It follows by elementary Hilbert $C^*$-module theory that $X$ is a Hilbert $C^*$-module over $B$ and the inclusion $\io : \C X \to X$ is completely bounded since
\[
\| \inn{ \io(\xi),\io(\xi)} \|_\infty = \| \io( \inn{\xi,\xi} ) \|_\infty \leq C \cd \| \inn{\xi,\xi} \|_1 \leq C \cd \| \xi \|_1^2
\]
for all $\xi \in M_m(\C X)$, $m \in \nn$ (for a uniform constant $C > 0$). The left action of $\C A$ on $\C X$ induces a $*$-homomorphism $A \to \sL(X)$ by Assumption \ref{a:opecor} $(3)$ and Definition \ref{d:herope} $(4)$, where $\sL(X)$ denotes the $C^*$-algebra of bounded adjointable operators on $X$.
\end{proof}

\begin{dfn}\label{d:coudeg}
An operator $*$-correspondence $\C X$ from $\C A$ to $\C B$ is said to be \emph{countably generated} (resp. \emph{non-degenerate}) when the $C^*$-completion $X$ is countably generated (resp. non-degenerate) as a $C^*$-correspondence from $A$ to $B$. Thus, when there exists a sequence $\{ \xi_n \}$ in $X$ such that
\[
\T{span}_{\cc}\{ \xi_n \cd b \mid n \in \nn \, , \, \, b \in \C B \} \su X
\]
is norm-dense in $X$ (resp. when $\T{span}_{\cc}\{ a \cd \xi \mid a \in A \, , \, \, \xi \in X\}$ is norm-dense in $X$).
%
\end{dfn}
%

The following definition of a differentiable structure on a $C^*$-correspondence plays an important role in \cite{Kaa:UKM}:

\begin{dfn}\label{d:cordif}
A $C^*$-correspondence $X$ from $A$ to $B$ is \emph{differentiable from $\C A$ to $\C B$} when there exists a sequence $\{ \xi_n \}$ in $X$ such that:
\begin{enumerate}
\item $\T{span}_{\cc}\big\{ \xi_n \cd b \mid b \in B \, , \, \, n \in \nn \big\}$ is dense in $X$;
\item $\inn{\xi_n, (a + \la) \cd \xi_m } \in \C B$ for all $a \in \C A$, $\la \in \cc$ and all $n,m \in \nn$;
\item The sequence of finite matrices
\[
\Big\{  \sum_{n,m = 1}^N \inn{\xi_n, (a + \la)\cd \xi_m} e_{nm} \Big\}_{N = 1}^\infty
\]
is a Cauchy sequence in $K_{\C B}$ for all $a \in \C A$, $\la \in \cc$;
\item The linear map $\tau : \C A \to K_{\C B}$, $a \mapsto \sum_{n,m = 1}^\infty \inn{\xi_n, a \cd \xi_m} e_{nm}$ is completely bounded.
\end{enumerate}
In this case, the sequence $\{ \xi_n\}$ is referred to as a \emph{differentiable generating sequence}.
\end{dfn}

The next result creates a link between operator $*$-correspondences and the above notion of differentiability for $C^*$-correspondences.

\begin{lemma}\label{l:diflinI}
Let $\C X$ be a countably generated operator $*$-correspondence from $\C A$ to $\C B$. Then there exists a sequence $\{ \xi_n\}$ of elements in $\C X$ such that
\begin{enumerate}
\item $\{ \xi_n \cd b \mid b \in \C B \, , \, \, n \in \nn \big\}$ is norm-dense in the $C^*$-completion $X$;
\item The sequence $\big\{ \sum_{n = 1}^N \xi_n \cd e_{1n} \big\}$ is a Cauchy sequence in $\ell^2(\nn,\C X)^t$. 
\end{enumerate}

Furthermore, the $C^*$-completion of $\C X$ is differentiable from $\C A$ to $\C B$ and any sequence $\{ \xi_n \}$ satisfying $(1)$ and $(2)$ yields a differentiable generating sequence $\{\io(\xi_n)\}$ (where $\io : \C X \to X$ is the inclusion).
\end{lemma}
\begin{proof}
Since $\C X$ is countably generated we may choose a sequence $\{ \eta_n \}$ in $\C X$ such that $\T{span}_{\cc}\{ \eta_n \cd b \mid b \in \C B \, , \, \, n \in \nn \big\}$ is norm-dense in $X$. Without loss of generality, we may assume that $\eta_n \neq 0$ for all $n \in \nn$. Define
\[
\xi_n := \eta_n \cd \frac{1}{n \cd \| \eta_n \|_{\C X}} \q n \in \nn
\]
It is clear that $(1)$ holds for the sequence $\{ \xi_n\}$. To prove that $(2)$ holds, we let $M > N \geq 1$ be given and notice that
\[
\sum_{n = N + 1}^M \xi_n \cd e_{1n} = \Big( \sum_{n = N + 1}^M \frac{1}{n} \cd e_{1n} \Big)  \cd \Big( \sum_{n = N + 1}^M \eta_n \cd \frac{1}{ \| \eta_n\|_{\C X}} \cd e_{nn} \Big)
\]
Since $\C X$ is an operator space, this implies that 
\[
\| \sum_{n = N + 1}^M \xi_n \cd e_{1n} \|_{\C X} \leq \| \sum_{n = N + 1}^M \frac{1}{n} \cd e_{1n} \|_{\cc}
= \sqrt{ \sum_{n = N+1}^M \frac{1}{n^2} }
\]
Thus $(2)$ holds for the sequence $\{ \xi_n \}$.

Let $\{ \xi_n\}$ be an arbitrary sequence satisfying $(1)$ and $(2)$. It is clear that $(1)$ and $(2)$ of Definition \ref{d:cordif} holds for the sequence $\{ \io(\xi_n) \}$ in $X$. To prove condition $(3)$ in Definition \ref{d:cordif} we notice that
\[
\sum_{n,m = 1}^N \inn{\xi_n, (a + \la) \cd \xi_m}_{\C X} \cd e_{nm} = \inn{ \sum_{n = 1}^N \xi_n \cd e_{1n}, \sum_{m = 1}^N (a + \la) \cd \xi_m \cd e_{1m}}_{\ell^2(\nn,\C X)^t}
\]
for all $N \in \nn$, $a \in \C A$ and $\la \in \cc$, see \eqref{eq:rowpair}. But this proves $(3)$ since both $\big\{ \sum_{n = 1}^N \xi_n \cd e_{1n} \big\}_{N = 1}^\infty$ and $\big\{ \sum_{m = 1}^N (a + \la) \cd \xi_m \cd e_{1m} \big\}_{N = 1}^\infty$ are Cauchy sequences in $\ell^2(\nn,\C X)^t$.

To prove condition $(4)$ in Definition \ref{d:cordif} we let $\xi \in \ell^2(\nn,\C X)^t$ denote the limit $\xi := \lim_{N \to \infty} \sum_{n = 1}^N \xi_n \cd e_{1n}$. The identities
\[
\tau(a) = \sum_{n,m = 1}^\infty \inn{\xi_n, a \cd \xi_m}_{\C X} \cd e_{nm} = \inn{\xi, a \cd \xi}_{\ell^2(\nn,\C X)^t} \q a \in \C A
\]
then imply $(4)$.
\end{proof}

\section{Equivalence relations on operator $*$-correspondences}\label{s:equrel}
Recall that all operator $*$-algebras satisfy Assumption \ref{a:opealg} and that all operator $*$-correspondences satisfy Assumption \ref{a:opecor}. In this context, it is clear what it means for two operator $*$-correspondences $\C X$ and $\C X'$ (both from $\C A$ to $\C B$) to be \emph{unitarily equivalent}. Indeed, this happens when there exist a unitary isomorphism $U : X \to X'$ between the two $C^*$-completions such that both $U$ and $U^*$ restrict to completely bounded maps $U : \C X \to \C X'$ and $U^* : \C X' \to \C X$. It turns out however that this notion of equivalence is too strict for our ($KK$-theoretic) purposes and we have therefore settled on the following much more flexible relation:

\begin{dfn}\label{d:difuni}
Let $\C X$ and $\C X'$ be two operator $*$-correspondences from $\C A$ to $\C B$. We say that $\C X$ and $\C X'$ are \emph{in duality} when there exists a unitary operator $U : X \to X'$ (where $X$ and $X'$ are the $C^*$-completions) such that:
\begin{enumerate}
\item $U (a \cd \xi) = a \cd  (U \xi)$ for all $a \in A$ and $\xi \in X$;
\item $\inn{U (\xi), \eta}_{X'} \in \C B$ for all $\xi \in \C X$ and $\eta \in \C X'$ (suppressing the inclusions into the $C^*$-completions);
\item There exists a constant $C > 0$ such that
\[
\big\| \inn{U(\xi), \eta}_{X'} \big\|_{\C B} \leq C \cd \| \xi \|_{\C X} \cd \| \eta \|_{\C X'}
\]
for all finite matrices $\xi \in M_m(\C X)$, $\eta \in M_m(\C X')$ and $m \in \nn$.
\end{enumerate}
\end{dfn}

Remark that the unitary operator $U : X \to X'$ in the above definition is not even required to map $\C X$ into $\C X'$.

\begin{lemma}
The relation ``in duality'' is reflexive and symmetric.
\end{lemma}
\begin{proof}
Reflexivity follows since the inner product $\inn{\cd,\cd} : \C X \ti \C X \to \C B$ satisfies the Cauchy-Schwarz inequality for any operator $*$-correspondence $\C X$.

Suppose thus that $\C X$ and $\C X'$ are in duality via the unitary operator $U : X \to X'$. Then $\C X'$ and $\C X$ are in duality via the unitary operator $U^* : X' \to X$.
\end{proof}

We ignore for the moment whether the relation ``in duality'' is transitive or not and we therefore make the following:

\begin{dfn}
The equivalence relation $\sim_d$ on operator $*$-correspondences from $\C A$ to $\C B$ is defined by
\[
\begin{split}
& (\C X \sim_d \C X') \lrar \Big( \T{ there exist } n \in \nn_0 \T{ and operator $*$-correspondences } \\ 
& \qqq \q \C X = \C X_0, \C X_1, \ldots,\C X_{n+1} = \C X'\T{ such that } \\ 
& \qqq \qq \C X_j \T{ and } \C X_{j + 1} \T{ are in duality for all } j \in \{0,\ldots,n\} \Big)
\end{split}
\]
By a slight abuse of language we will sometimes say that two operator $*$-correspondences $\C X$ and $\C X'$ are \emph{in duality} whenever $\C X \sim_d \C X'$.
\end{dfn}

The following situation often occurs:

\begin{lemma}\label{l:comboudua}
Let $\C X$ and $\C X'$ be two operator $*$-correspondences (both from $\C A$ to $\C B$). Suppose that $u : \C X \to \C X'$ is a completely bounded bimodule map which induces a unitary isomorphism $U : X \to X'$ of $C^*$-correspondences. Then we have that $\C X \sim_d \C X'$.
\end{lemma}

We remark that the conditions in Lemma \ref{l:comboudua} \emph{does not imply} that the adjoint $U^* : X' \to X$ restricts to a completely bounded map from $\C X'$ to $\C X$ and the two operator $*$-correspondences $\C X$ and $\C X'$ need therefore not be completely bounded isomorphic. 

\section{Algebraic operations on operator $*$-correspondences}\label{s:algope}
We are now going to introduce direct sums and interior tensor products of operator $*$-correspondences. The standard conditions on operator $*$-algebras and operator $*$-correspondences stated in Assumption \ref{a:opealg} and Assumption \ref{a:opecor} are in effect throughout this section.

\subsection{Direct sums}
Let $\C X$ and $\C Y$ be two operator $*$-correspondences both from $\C A$ to $\C B$.
\medskip

We let $\C X \op \C Y$ denote the direct sum of $\C A$-$\C B$-bimodules. This direct sum comes equipped with the inner product defined by $\inn{ (x_1,y_1) , (x_2,y_2)}_{\C X \op \C Y} :=  \inn{x_1,x_2}_{\C X} + \inn{y_1,y_2}_{\C Y}$. To define the matric norms for $\C X \op \C Y$ we let $p_1 : \C X \op \C Y \to \C X$ and $p_2 : \C X \op \C Y \to \C Y$ denote the canonical projections. These projections then induce maps $p_1 : M_m(\C X \op \C Y) \to M_m(\C X)$ and $p_2 : M_m(\C X \op \C Y) \to M_m(\C Y)$ at the level of finite matrices. For each $m \in \nn$, we define the norm on $M_m(\C X \op \C Y)$ by
\begin{equation}\label{eq:dirdif}
\| z \|_{\C X \op \C Y} := \sqrt{2} \cd \T{max}\{ \| p_1(z) \|_{\C X} , \| p_2(z) \|_{\C Y} \} \q \forall z \in M_m(\C X \op \C Y)
\end{equation}

\begin{lemma}
The direct sum $\C X \op \C Y$ is an operator $*$-correspondence from $\C A$ to $\C B$ and it satisfies Assumption \ref{a:opecor}. The $C^*$-completion of $\C X \op \C Y$ agrees with the direct sum $X \op Y$ of the $C^*$-completions of $\C X$ and $\C Y$.
\end{lemma}
\begin{proof}
As an operator space $\C X \op \C Y$ is just a rescaled version of the $\infty$-direct sum of operator spaces, see \cite[\S 1.2.17]{BlMe:OAM}. Furthermore, since $p_1 : \C X \op \C Y \to \C X$ and $p_2 : \C X \op \C Y \to \C Y$ are bimodule homomorphisms we obtain that $\C X \op \C Y$ is indeed an operator $\C A$-$\C B$-bimodule.

For each $z, w \in M_m(\C X \op \C Y)$, $m \in \nn$, we have that $\inn{z,w}_{\C X \op \C Y} = \inn{p_1(z),p_1(w)}_{\C X} + \inn{p_2(z),p_2(w)}_{\C Y}$. It thus follows that
\[
\| \inn{z,w} \|_{\C B} \leq \| p_1(z) \|_{\C X} \cd \| p_1(w) \|_{\C X} + \| p_2(z) \|_{\C Y} \cd \| p_2(w) \|_{\C Y}
\leq  \| z \|_{\C X \op \C Y} \cd \| w \|_{\C X \op \C Y}
\]
where the last inequality holds because of the ``extra'' factor $\sqrt{2}$ in \eqref{eq:dirdif}. This shows that $\C X \op \C Y$ is indeed an operator $*$-correspondence from $\C A$ to $\C B$. 

We leave it to the reader to verify that the conditions in Assumption \ref{a:opecor} hold for $\C X \op \C Y$ and that the $C^*$-completion agrees with the direct sum $X \op Y$ of $C^*$-correspondences.
\end{proof}

\begin{remark}\label{r:assdif}
Let $\C Z$ be an extra operator $*$-correspondence from $\C A$ to $\C B$. The obvious map $(\C X \op \C Y) \op \C Z \to \C X \op ( \C Y \op \C Z )$ is \emph{not} a complete isometry but it is (of course) a completely bounded isomorphism which induces a unitary operator at the level of $C^*$-completions. It therefore holds that $(\C X \op \C Y) \op \C Z \sim_d \C X \op ( \C Y \op \C Z )$.
\end{remark}

\subsection{Interior tensor products}
Let $\C X$ and $\C Y$ be operator $*$-correspondences from $\C A$ to $\C B$ and from $\C B$ to $\C C$, respectively.
\medskip

We start by forming the Haagerup tensor product $\C X \wot \C Y$ of the operator spaces $\C X$ and $\C Y$. To be explicit, we have the matrix norm $\| \cd \|_{\C X \wot \C Y} : M_m(\C X \ot \C Y) \to [0,\infty)$ on the algebraic tensor product of $\C X$ and $\C Y$, defined by
\[
\| z \|_{\C X \wot \C Y} := \inf\big\{ \| x \|_{\C X} \cd \| y \|_{\C Y} \mid 
z = x \ot y \, , \, \, k \in \nn \, , \, \, x \in M_{m,k}(\C X) \, , \, \, y \in M_{k,m}(\C Y) \big\}
\]
where $(x \ot y)_{ij} := \sum_{l = 1}^k x_{il} \ot y_{lj}$ for all $i,j \in \{1,\ldots,m\}$ whenever $x \in M_{m,k}(\C X)$ and $y \in M_{k,m}(\C Y)$ for some $k \in \nn$. The Haagerup tensor product $\C X \wot \C Y$ is obtained as the completion of the algebraic tensor product $\C X \ot \C Y$ with respect to the above norm on $M_1(\C X \ot \C Y) \cong \C X \ot \C Y$, see \cite[\S 1.5.4]{BlMe:OAM} and \cite{PaSm:MTO}. 

The Haagerup tensor product becomes an operator $\C A$-$\C C$-bimodule with left and right action induced by
\[
a \cd (x \ot y) := (a \cd x) \ot y \q \T{and} \q (x \ot y) \cd c := x \ot (y \cd c)
\]
for all $a \in \C A$ , $x \in \C X$, $y \in \C Y$ and $c \in \C C$.

To define an inner product on $\C X \wot \C Y$ we introduce the pairing $\inn{\cd,\cd}_{\C X \ot \C Y} : \C X \ot \C Y \ti \C X \ot \C Y \to \C C$ on the algebraic tensor product:
\[
\inn{ x_0 \ot y_0, x_1 \ot y_1}_{\C X \ot \C Y} := \inn{y_0, \inn{x_0,x_1}_{\C X} \cd y_1}_{\C Y} \q x_0 , x_1 \in \C X \, , \, \, y_0, y_1 \in \C Y 
\]
This pairing is compatible with the bimodule structure on the Haagerup tensor product in the sense that the conditions $(1)$-$(4)$ of Definition \ref{d:herope} hold.

\begin{lemma}\label{l:causch}
The pairing $\inn{\cd,\cd}_{\C X \ot \C Y} : \C X \ot \C Y \ti \C X \ot \C Y \to \C C$ satisfies the Cauchy-Schwarz inequality:
\[
\| \inn{z_0,z_1}_{\C X \ot \C Y} \|_{\C C} \leq \|z_0\|_{\C X \wot \C Y} \cd \| z_1 \|_{\C X \wot \C Y}
\]
for all $z_0,z_1 \in M_m(\C X \ot \C Y)$ and $m \in \nn$.
\end{lemma}
\begin{proof}
Let $z_0, z_1 \in M_m(\C X \ot \C Y)$ for some $m \in \nn$. A direct computation shows that
\[
\inn{z_0,z_1}_{\C X \ot \C Y} = \inn{y_0, \inn{x_0,x_1}_{\C X} \cd y_1}_{\C Y}
\]
whenever $z_0 = x_0 \ot y_0$ and $z_1 = x_1 \ot y_1$ for some $x_0, x_1 \in M_{m,k}(\C X)$ and $y_0, y_1 \in M_{k,m}(\C Y)$. Using the fact that $\C X$ and $\C Y$ are operator $*$-correspondences we may thus estimate as follows:
\[
\| \inn{z_0, z_1}_{\C X \ot \C Y} \|_{\C C} \leq \| y_0 \|_{\C Y} \cd \| x_0 \|_{\C X} \cd \| x_1 \|_{\C X} \cd \|y_1 \|_{\C Y}
\]
This implies the result of the lemma.
\end{proof}

It follows from the above lemma that $\C X \wot \C Y$ is an operator $*$-correspondence from $\C A$ to $\C C$ in the sense of Definition \ref{d:herope} (with inner product $\inn{\cd, \cd}_{\C X \wot \C Y}$ induced by $\inn{\cd,\cd}_{\C X \ot \C Y}$). It might however happen that $\C X \wot \C Y$ is too big to satisfy the second condition of Assumption \ref{a:opecor}. We thus define the subset 
\[
\C N \su \C X \wot \C Y \q \C N := \{ z \in \C X \wot \C Y \mid \inn{z,z}_{\C X \wot \C Y} = 0 \}
\]

\begin{lemma}\label{l:submod}
The subset $\C N \su \C X \wot \C Y$ is a closed $\C A$-$\C C$-bisubmodule. Moreover, we have that $\inn{z,w}_{\C X \wot \C Y} = 0$ whenever $z \in \C N$ and $w \in \C X \wot \C Y$.
\end{lemma}
\begin{proof}
It is clear that $\C N \su \C X \wot \C Y$ is a closed right $\C C$-submodule. To prove the remaining claims of the lemma we let $X \hot_B Y$ denote the interior tensor product of the $C^*$-completions, see \cite[Theorem 5.9]{Rie:IRA} and \cite[Chapter 4]{Lan:HCM}. We then remark that the canonical map $\C X \ot \C Y \to X \hot_B Y$ extends to a completely bounded map $\io : \C X \wot \C Y \to X \hot_B Y$ which is compatible with the inner products and the bimodule structures. This is indeed a consequence of Lemma \ref{l:causch} and the complete boundedness of the inclusion $\io : \C C \to C$. Let now $z \in \C N$, $a \in \C A$ and $w \in \C X \wot \C Y$ be given. We have that
\[
\begin{split}
0 & \leq \io( \inn{a \cd z, a \cd z}_{\C X \wot \C Y} ) = \inn{ a \cd \io(z), a \cd \io(z)}_{X \hot_B Y} \\
& \leq \| a \|_\infty^2 \cd \inn{\io(z),\io(z)}_{X \hot_B Y} = 0
\end{split}
\]
and furthermore that
\[
\| \io( \inn{z,w}_{\C X \wot \C Y} ) \|_\infty \leq \| \inn{\io(z),\io(z)}_{X \hot_B Y} \|_\infty^{1/2} \cd \| \inn{\io(w),\io(w)}_{X \hot_B Y} \|_\infty^{1/2} = 0  
\]
Since the inclusion $\io : \C C  \to C$ is injective this proves the lemma.
\end{proof}

By the above lemmas we obtain that the quotient 
\[
\C X \hot_{\C B} \C Y := (\C X \wot \C Y)/\C N
\]
inherits a well-defined operator $*$-correspondence structure from the Haagerup tensor product $\C X \wot \C Y$. Notice that the relevant matrix norms $\| \cd \|_{\C X \hot_{\C B} \C Y} : M_m(\C X \hot_{\C B} \C Y) \cong M_m(\C X \wot \C Y) / M_m(\C N) \to [0,\infty)$, $m \in \nn$, are the quotient norms. As in the proof of Lemma \ref{l:submod} we let $X \hot_B Y$ denote the interior tensor product of the $C^*$-completions of $\C X$ and $\C Y$, see \cite[Chapter 4]{Lan:HCM}. The next result is now immediate:

\begin{prop}\label{p:tendif}
The operator $*$-correspondence $\C X \hot_{\C B} \C Y$ from $\C A$ to $\C C$ satisfies the conditions of Assumption \ref{a:opecor} and the $C^*$-completion agrees with the interior tensor product $X \hot_B Y$ of the $C^*$-completions of $\C X$ and $\C Y$.
\end{prop}

We refer to $\C X \hot_{\C B} \C Y$ as the \emph{interior tensor product} of the operator $*$-correspondences $\C X$ and $\C Y$.


\section{The Morita monoid}\label{s:morita}
Our aim is now to summarize and improve our work so far on operator $*$-correspondences by defining a category $\mathfrak{M}$ where the objects are operator $*$-algebras and the morphisms are equivalence classes of operator $*$-correspondences satisfying a few side-conditions. Most importantly these side-conditions include a compactness condition that makes it possible to relate our category to $KK$-theory. This category will also allow us to introduce our notion of Morita equivalence for operator $*$-algebras.

We recall that all operator $*$-algebras and all operator $*$-correspondences satisfy Assumption \ref{a:opealg} and Assumption \ref{a:opecor}. 
%
\medskip

For a Hilbert $C^*$-module $X$ we let $\sL(X)$ and $\sK(X)$ denote the $C^*$-algebras of bounded adjointable operators and compact operators on $X$, respectively. Both of these $C^*$-algebras are equipped with the operator norm.
 
\begin{dfn}\label{d:compact}
We will say that an operator $*$-correspondence $\C X$ from $\C A$ to $\C B$ is \emph{compact} when $\C X$ is countably generated and non-degenerate in the sense of Definition \ref{d:coudeg} and when the induced left action $\pi : A \to \sL(X)$ on the $C^*$-completion of $\C X$ factorizes through the Hilbert $C^*$-module compacts $\sK(X) \su \sL(X)$, thus when $\pi(a) \in \sK(X)$ for all $a \in A$.
\end{dfn}


\begin{dfn}
Let $\C A$ and $\C B$ be operator $*$-algebras. By the \emph{Morita monoid} from $\C A$ to $\C B$ we will understand the collection of compact operator $*$-correspondences from $\C A$ to $\C B$ modulo the equivalence relation $\sim_d$ generated by ``in duality'' (see Definition \ref{d:difuni}). The Morita monoid is denoted by $M(\C A, \C B)$.
\end{dfn}

The monoid structure of $M(\C A,\C B)$ is explained by the following elementary:

\begin{lemma}
The direct sum of compact operator $*$-correspondences provides $M(\C A, \C B)$ with the structure of a commutative monoid with trivial element given by (the class of) the trivial operator $*$-correspondence.
\end{lemma}

The composition in our category $\mathfrak{M}$ will be implemented by the interior tensor product of operator $*$-correspondences:

\begin{lemma}
The interior tensor product of compact operator $*$-correspondences induces a bilinear and associative pairing
\[
\hot_{\C B} : M(\C A, \C B) \ti M(\C B, \C C) \to M(\C A,\C C)
\]
\end{lemma}
\begin{proof}
Let $\C X$ and $\C Y$ be compact operator $*$-correspondences from $\C A$ to $\C B$ and from $\C B$ to $\C C$, respectively. By Proposition \ref{p:tendif} their interior tensor product $\C X \hot_{\C B} \C Y$ is an operator $*$-correspondence from $\C A$ to $\C C$ with $C^*$-completion given by the interior tensor product $X \hot_B Y$ of $C^*$-correspondences. Since both $X$ and $Y$ are countably generated and non-degenerate $C^*$-correspondences it follows that $X \hot_B Y$ is countably generated and non-degenerate as well. Let $\pi \hot 1 : A \to \sL(X \hot_B Y)$ denote the left action of $A$ on the interior tensor product of the $C^*$-completions of $\C X$ and $\C Y$. The fact that $\pi(a) \hot 1$, $a \in A$, is a compact operator on $X \hot_B Y$ follows by \cite[Proposition 4.7]{Lan:HCM} since both $\C X$ and $\C Y$ are compact. This shows that $\C X \hot_{\C B} \C Y$ determines an element in $M(\C A,\C C)$.
%
\medskip

Let $\C X'$ be a compact operator $*$-correspondence that is in duality with $\C X$ through a unitary operator $U : X \to X'$. The interior tensor products $X \hot_B Y$ and $X' \hot_B Y$ are then unitarily isomorphic via the unitary operator $U \hot 1 : X \hot_B Y \to X' \hot_B Y$. To show that condition $(2)$ and $(3)$ of Definition \ref{d:difuni} are satisfied as well, we let $z \in M_m(\C X \ot \C Y)$ and $w \in M_m(\C X' \ot \C Y)$ for some $m \in \nn$ be given. Choose elements $x \in M_{m,k}(\C X)$, $x' \in M_{m,k}(\C X')$, and $y_0, y_1 \in M_{k,m}(\C Y)$ such that $z = x \ot y_0$ and $w = x' \ot y_1$. We have that
\[
\inn{ (U \ot 1)(z), w }_{X' \hot_B Y} = \inn{y_0, \inn{U(x), x'}_{X'}\cd y_1}_{\C Y} \in \C C
\]
Furthermore, there exists a constant $C > 0$ (which is independent of our previous choices) such that
\[
\begin{split}
\| \inn{ (U \ot 1)(z), w }_{X'\hot_B Y} \|_{\C C} 
& \leq 
\| y_0 \|_{\C Y} \cd \| y_1 \|_{\C Y} \cd \| \inn{U(x), x'}_{X'} \|_{\C B} \\
& \leq \| y_0 \|_{\C Y} \cd \| y_1 \|_{\C Y} \cd C \cd \| x \|_{\C X} \cd \| x' \|_{\C X'}
\end{split}
\]
These identities and estimates imply that $\C X \hot_{\C B} \C Y \sim_d \C X' \hot_{\C B} \C Y$. A similar argument shows that $\C X \hot_{\C B} \C Y \sim_d \C X \hot_{\C B} \C Y'$ whenever $\C Y \sim_d \C Y'$ and we conclude that the interior tensor product of operator $*$-correspondences yields a well-defined map $\hot_{\C B} : M(\C A,\C B) \ti M(\C B,\C C) \to M(\C A,\C C)$.
\medskip

To prove that the pairing $\hot_{\C B} : M(\C A,\C B) \ti M(\C B,\C C) \to M(\C A,\C C)$ is bilinear we let $[\C X], [\C X'] \in M(\C A,\C B)$ and $[\C Y] \in M(\C B,\C C)$ be given. The projections $p_1 : \C X \op \C X' \to \C X$ and $p_2 : \C X \op \C X' \to \C X'$ then induce the completely bounded bimodule map:
\[
u := \big( (p_1 \hot 1), (p_2 \hot 1) \big) : (\C X \op \C X') \hot_{\C B} \C Y \to (\C X \hot_{\C B} \C Y) \op (\C X' \hot_{\C B} \C Y)
\]
Since $u$ extends to a unitary isomorphism of $C^*$-correspondences $U : (X \op X') \hot_B Y \to (X \hot_B Y) \op (X' \hot_B Y)$ we conclude by Lemma \ref{l:comboudua} that
\[
(\C X \op \C X') \hot_{\C B} \C Y \sim_d (\C X \hot_{\C B} \C Y) \op (\C X' \hot_{\C B} \C Y)
\]
A similar argument shows that $\hot_{\C B} : M(\C A,\C B) \ti M(\C B,\C C) \to M(\C A,\C C)$ is linear in the second variable as well.
\medskip

Let $[\C X] \in M(\C A,\C B)$, $[\C Y] \in M(\C B,\C C)$ and $[\C Z] \in M(\C C,\C D)$ be given. The associativity of our pairing follows since the canonical map $(\C X \ot \C Y) \ot \C Z \to \C X \ot (\C Y \ot \C Z)$ induces a completely isometric isomorphism
\[
u : (\C X \hot_{\C B} \C Y) \hot_{\C C} \C Z \to \C X \hot_{\C B} (\C Y \hot_{\C C} \C Z)
\]
at the level of interior tensor products of operator $*$-correspondences. Furthermore, this latter completely isometric isomorphism extends to a unitary isomorphism $U : (X \hot_B Y) \hot_C Z \to X \hot_B (Y \hot_C Z)$ of interior tensor products of $C^*$-correspondences. By Lemma \ref{l:comboudua} we thus have that $[\C X \hot_{\C B} \C Y] \hot_{\C C} [\C Z] = [\C X] \hot_{\C B} [\C Y \hot_{\C C} \C Z]$ in the Morita monoid $M(\C A,\C D)$.
\end{proof}

Recall from Subsection \ref{ss:rowcol} that any operator $*$-algebra $\C A$ may be considered as an operator $*$-correspondence $\C A$ from $\C A$ to $\C A$. 


\begin{prop}
The operator $*$-correspondence $\C A$ determines an element in $M(\C A,\C A)$ and we have the identities
\[
[\C A] \hot_{\C A} [\C X] = [\C X] \q \T{and} \q [\C Y] \hot_{\C A} [\C A] = [\C Y]
\]
for all $[\C X] \in M(\C A,\C B)$ and $[\C Y] \in M(\C B,\C A)$.
\end{prop}
\begin{proof}
The $C^*$-completion of the operator $*$-correspondence $\C A$ agrees with the $C^*$-algebra $A$ considered as a $C^*$-correspondence from $A$ to $A$. Since $A \cd A = A$ and $A$ is $\si$-unital by the standing Assumption \ref{a:opealg} we obtain that $A$ is non-degenerate and countably generated. Moreover, $A$ acts from the left on $A$ by compact operators and we may thus conclude that $\C A$ determines an element in $M(\C A,\C A)$.
%

Let $\C X$ be a compact operator $*$-correspondence from $\C A$ to $\C B$. The left action of $\C A$ on $\C X$ induces a completely bounded bimodule map $u : \C A \hot_{\C A} \C X \to \C X$, which in turn, by the non-degeneracy of $\C X$, induces a unitary isomorphism $U : A \hot_A X \to X$ of $C^*$-correspondences. By Lemma \ref{l:comboudua} we thus have the identity $[\C A] \hot_{\C A} [\C X] = [\C X]$ in the Morita monoid $M(\C A, \C B)$. A similar argument shows that $[\C Y] \hot_{\C A} [\C A] = [\C Y]$ in the Morita monoid $M(\C B,\C A)$ whenever $\C Y$ is a compact operator $*$-correspondence from $\C B$ to $\C A$.
\end{proof}

The above results imply that we have a category $\mathfrak{M}$ with objects consisting of operator $*$-algebras $\C A,\C B,\C C, \ldots$ and with morphisms $\T{Mor}_{\mathfrak{M}}(\C A,\C B) := M(\C A,\C B)$ given by the Morita monoid. The composition is given by the interior tensor product $\hot_{\C B} : M(\C A,\C B) \ti M(\C B,\C C) \to M(\C A,\C C)$ and the identity morphism (from $\C A$ to $\C A$) is given by the class $[\C A] \in M(\C A,\C A)$.
\medskip

We are now ready for the main definition of this part of the paper (Section \ref{s:opealg} to Section \ref{s:morita}):

\begin{dfn}\label{d:morequ}
We say that two operator $*$-algebras $\C A$ and $\C B$ are \emph{Morita equivalent} when there exists an invertible morphism $[\C X] \in M(\C A,\C B)$. In this case, we write $\C A \sim_m \C B$.
\end{dfn}

It is clear that the relation $\sim_m$ is indeed an equivalence relation. 
\medskip

We end this section by relating the above notion of Morita equivalence for operator $*$-algebras to Rieffel's notion of Morita equivalence for $C^*$-algebras:

\begin{prop}
Suppose that $\C A$ and $\C B$ are Morita equivalent operator $*$-algebras. Then the $C^*$-completions $A$ and $B$ are Morita equivalent in the sense of Rieffel, see \cite[Definition 8.17]{Rie:MEC}.
\end{prop}
\begin{proof}
By Definition \ref{d:morequ} there exist compact operator $*$-correspondences $\C X$ and $\C Y$ from $\C A$ to $\C B$ and from $\C B$ to $\C A$, respectively, such that $\C X \hot_{\C B} \C Y \sim_d \C A$ and $\C Y \hot_{\C A} \C X \sim_d \C B$. This implies that the $C^*$-completion of $\C X \hot_{\C B} \C Y$ is unitarily isomorphic to the $C^*$-completion of $\C A$ and that the $C^*$-completion of $\C Y \hot_{\C A} \C X$ is unitarily isomorphic to the $C^*$-completion of $\C B$. An application of Proposition \ref{p:tendif} then shows that the interior tensor products $X \hot_B Y$ and $Y \hot_A X$ are unitarily isomorphic to $A$ and $B$, respectively. This proves the proposition.
\end{proof}

\section{Unbounded modular cycles}\label{s:unbmod}
What we would like to do at this moment is to write down a pairing between compact operator $*$-correspondences (thus representatives for elements in the Morita monoid) and unbounded Kasparov modules in the sense of Baaj and Julg, \cite{BaJu:TBK}. It turns out however that the class of unbounded Kasparov modules is not quite flexible enough to perform such a pairing in an explicit way. For this reason we were led to introduce a new class of unbounded cycles in \cite{Kaa:UKM} and we will now briefly recall their definition and explain how these unbounded cycles relate to usual unbounded Kasparov modules and to twisted spectral triples in the sense of Connes and Moscovici, \cite{CoMo:TST}.
\medskip

Let $\C A$ be an operator $*$-algebra satisfying Assumption \ref{a:opealg} and let $B$ be a $\si$-unital $C^*$-algebra.
\medskip

For any bounded adjointable operator $T : X \to X$ defined on a Hilbert $C^*$-module $X$ we let $C^*(T) \su \sL(X)$ denote the smallest $C^*$-subalgebra containing $T : X \to X$. It will not be required that the unit lies in $C^*(T)$.
\medskip

We recall that an unbounded selfadjoint operator $D : \sD(D) \to X$, which is densely defined on a Hilbert $C^*$-module $X$, is said to be \emph{regular} when $D \pm i : \sD(D) \to X$ are surjective, see for example \cite{Baa:MNB} and \cite[Chapter 9]{Lan:HCM}.
\medskip

We let $\wit{\C A}$ denote the unitalization of $\C A$ (considered as a $*$-algebra). For the purposes of this paper it suffices to equip the unitalization $\wit{\C A}$ with (a rescaled version of) the operator space structure coming from the vector space isomorphism $\wit{\C A} \cong \C A \op \cc$. Thus, letting $p_1 : M_m(\wit{\C A}) \to M_m(\C A)$ and $p_2 : M_m(\wit{\C A}) \to M_m(\C \cc)$, $m \in \nn$, be the linear maps obtained by applying the canonical projections entry-wise, we define
\[
\| x \|_{\wit{\C A}} := 3 \cd \T{max}\{ \| p_1(x) \|_{\C A} \, , \, \, \| p_2(x) \|_{\cc} \} \q x \in M_m(\wit{\C A}) \, , \, \, m \in \nn
\]
With these matrix norms $\wit{\C A}$ is an operator $*$-algebra but the inclusion $\C A \to \wit{\C A}$ is not a complete isometry ($\C A$ is however completely bounded isomorphic to its image in $\wit{\C A}$).
\medskip

The following definition can be found as \cite[Definition 3.1]{Kaa:UKM} (notice however that we are requiring the $C^*$-correspondence $X$ to be non-degenerate):

\begin{dfn}\label{d:unbkas}
An \emph{odd unbounded modular cycle} from $\C A$ to $B$ is a triple $(X,D,\De)$ where
\begin{enumerate}
\item $X$ is a countably generated and non-degenerate $C^*$-correspondence from $A$ to $B$ (with left action $\pi : A \to \sL(X)$);
\item $D : \sD(D) \to X$ is an unbounded selfadjoint and regular operator;
\item $\De : X \to X$ is a bounded positive operator with dense image,
\end{enumerate}
such that the following holds:
\begin{enumerate}
\item $\pi(a) \cd (i + D)^{-1} : X \to X$ is a compact operator for all $a \in A$;
\item The inclusion
\[
(\pi(a) + \la ) \De(\xi) \in \sD(D)
\]
holds for all $\xi \in \sD(D)$ and $(a,\la) \in \wit{\C A}$.
\item There exists a completely bounded linear map
\[
\rho_\De : \wit{\C A} \to \sL(X)
\]
such that
\[
\big( \De^{1/2} \rho_\De(a,\la) \De^{1/2} \big)(\xi) = \big( D (\pi(a) + \la ) \De - \De (\pi(a) + \la ) D \big)(\xi)
\]
for all $\xi \in \sD(D)$ and all $(a,\la) \in \wit{\C A}$ (remark that $\rho_\De$ need not be unital);
\item There exists a countable approximate identity $\{ V_n\}_{n = 1}^\infty$ for the $C^*$-algebra $C^*(\De)$ such that the sequence
\[
\big\{ V_n \pi(a) \big\}_{n=1}^\infty
\]
converges in operator norm to $\pi(a)$ for all $a \in A$.
\end{enumerate}
We will refer to $\De : X \to X$ as the \emph{modular operator} of our unbounded modular cycle.

An \emph{even unbounded modular cycle} from $\C A$ to $B$ is an odd unbounded modular cycle equipped with a $\zz/2\zz$-grading operator $\ga : X \to X$ such that 
\[
\ga \pi(a) = \pi(a) \ga \q \ga \De = \De \ga \q \T{and} \q \ga D = - D \ga
\]
for all $a \in A$. 
\end{dfn}

In relation to condition $(3)$ of Definition \ref{d:unbkas} we record the following useful:

\begin{lemma}\label{l:imahal}
Let $X$ and $Y$ be (right) Hilbert $C^*$-modules over $B$. Suppose that $\Phi : X \to Y$ is a bounded adjointable operator such that $\Phi^* : Y \to X$ has dense image. Define $\De := \Phi^* \Phi : X \to X$. Then $\T{Im}(\Phi^*) \su \T{Im}(\De^{1/2})$, the operator
\[
\De^{-1/2} \Phi^* : Y \to X
\]
is bounded adjointable and the adjoint $(\De^{-1/2} \Phi^*)^* : X \to Y$ is an extension of the unbounded operator 
\[
\Phi \De^{-1/2} : \T{Im}(\De^{1/2}) \to Y
\]
\end{lemma}
\begin{proof}
Remark that $\De = \Phi^* \Phi : X \to X$ is positive and has dense image (since $\Phi^* : Y \to X$ has dense image). The square root $\De^{1/2} : X \to X$ is therefore positive and has dense image too. It follows that the inverse $\De^{-1/2} : \T{Im}(\De^{1/2}) \to X$ in the formulation of the lemma makes sense and that it is a positive and regular unbounded operator.

Let $\eta \in X$ and notice that
\[
\inn{\Phi \De^{-1/2} (\De^{1/2}\eta), \Phi \De^{-1/2} (\De^{1/2} \eta)} = \inn{\De^{1/2}\eta,\De^{1/2}\eta}
\]
This shows that $\Phi \De^{-1/2} : \T{Im}(\De^{1/2}) \to Y$ is the restriction of a bounded operator $T : X \to Y$. Furthermore, we have that
\[
(\Phi^* T)(\De^{1/2} \eta) = (\Phi^* \Phi \De^{-1/2})(\De^{1/2} \eta) = \De^{1/2}(\De^{1/2} \eta)
\]
and we may conclude that $\Phi^* T = \De^{1/2} : X \to X$. We have thus proved that $\T{Im}(\Phi^*) \su \T{Im}(\De^{1/2})$.

The result of the lemma now follows from the identity $T^* = \De^{-1/2} \Phi^* : Y \to X$.
\end{proof}

In relation to condition $(4)$ of Definition \ref{d:unbkas} we remark that $\{ \De(\De + 1/n)^{-1} \}$ is a countable approximate identity for the $C^*$-algebra $C^*(\De)$. Moreover, if $(4)$ holds for some countable approximate identity $\{ V_n\}$ it holds for every countable approximate identity $\{ W_n\}$ for the $C^*$-algebra $C^*(\De)$. 
\medskip

We now present the link between unbounded Kasparov modules in the sense of Baaj and Julg and the above concept of an unbounded modular cycle. 


\begin{prop}
Let $(\C A,X,D)$ is an unbounded Kasparov module (from $A$ to $B$) with $X$ non-degenerate and countably generated. Suppose that the derivation $d : \C A \to \sL(X)$, given by $d(a)(\xi) = [D,\pi(a)](\xi)$ for all $\xi \in \sD(D)$, is completely bounded. Then $(X,D,1)$ is an unbounded modular cycle from $\C A$ to $B$ of the same parity as $(\C A,X,D)$.
\end{prop}
\begin{proof}
The only non-trivial part consists of noting that the completely bounded map $\rho_1 : \wit{\C A} \to \sL(X)$ is given by $\rho_1(a,\la) = d(a)$ for all $(a,\la) \in \wit{\C A}$.
\end{proof}

Let $(\C A,X,D)$ is an unbounded Kasparov module satisfying the conditions of the above proposition. To illustrate the extra flexibility present in the definition of an unbounded modular cycle we now consider an extra positive invertible bounded operator $G : X \to X$, which we require to be even when the unbounded Kasparov module $(\C A,X,D)$ is even. We assume that $G(\xi) \in \sD(D)$ for all $\xi \in \sD(D)$ and that $[D,G] : \sD(D) \to X$ extends to a bounded adjointable operator $d(G) \in \sL(X)$. 

In this case, it follows by \cite[Proposition 6.7]{Kaa:DAH} that $G D G :  \sD(D) \to X$ is an unbounded selfadjoint and regular operator, see also \cite[Example 2 and 3]{Wor:UAQ}. 

The next result should be compared with \cite[Lemma 2.1]{CoMo:TST}. It is motivated by the geometric situation where a Dirac operator on a compact spin manifold is replaced by the Dirac operator arising after passing to a conformally equivalent metric.

\begin{prop}\label{p:confeq}
The triple $(X, GD G, G^2)$ is an unbounded modular cycle from $\C A$ to $B$ of the same parity as $(\C A,X,D)$.
\end{prop}
\begin{proof}
We only need to verify conditions $(1)$-$(4)$ of Definition \ref{d:unbkas}. Condition $(4)$ is immediate since $C^*(G)$ is unital and condition $(2)$ follows by the construction of the operator $*$-algebra $\C A$. Indeed, all operators of the form $\pi(a) + \la : X \to X$, $(a,\la) \in \wit{\C A}$, preserve $\sD(GDG) = \sD(D)$ and so do $G : X \to X$. To verify condition $(1)$ we let $a \in A$. A direct computation shows that
\begin{equation}\label{eq:modures}
\begin{split}
& (i + G D G)^{-1} - G^{-1}(i + D)^{-1} G^{-1} \\
& \qqq = i G^{-1}(i + D)^{-1}(G - G^{-1})(i + G D G)^{-1}
\end{split}
\end{equation}
and hence that $\pi(a) G(i + G D G)^{-1} \in \sK(X)$. The resolvent identity then shows that
\[
\begin{split}
& \pi(a) (i + G D G)^{-1} \\
& \q = \pi(a) G (i + G DG)^{-1} G^{-1} \\ 
& \qqq - \pi(a) G (i + G D G)^{-1} d(G) (i+ G D G)^{-1}
\in \sK(X)
\end{split}
\]
and $(1)$ is proved. To verify $(3)$ we define the completely bounded linear map $\rho_{G^2} : \wit{\C A} \to \sL(X)$ by the formula
\[
\rho_{G^2}(a,\la) := d( G ) ( \pi(a) + \la) G + G d(a) G + G ( \pi(a) + \la) d(G)
\]
for all $(a,\la) \in \wit{\C A}$. For $(a,\la) \in \wit{\C A}$ and $\xi \in \sD(D) = \sD(G D G)$ we then have that
\[
(G \rho_{G^2}(a,\la) G)(\xi) = (G D G)( \pi(a) + \la) G^2(\xi) - G^2 (\pi(a) + \la) (G D G)(\xi)
\]
This proves the proposition.
\end{proof}

To obtain a partial converse to the above statements we consider the case where the $C^*$-algebra $A$ is unital:

\begin{prop}
Let $(X,D,\De)$ be an unbounded modular cycle from $\C A$ to $B$ and suppose that $A$ is unital. Then the modular operator $\De$ has a bounded inverse and the triple $(\C A,X, \De^{-1/2} D \De^{-1/2})$, with $\sD(\De^{-1/2} D \De^{-1/2}) := \sD(D)$, is an unbounded Kasparov module of the same parity as $(X,D,\De)$.
\end{prop}
\begin{proof}
Since $X$ is non-degenerate we obtain that $\pi(1) = 1$ and hence that $\De(\De + 1/n)^{-1} \to 1$ in operator norm. This implies that $\De$ does not have $0$ in the spectrum and thus that $\De^{-1}$ is a bounded adjointable operator. Since $\De$ preserves $\sD(D)$ and $[D,\De] : \sD(D) \to X$ extends to a bounded adjointable operator we have that $\De^{-1/2}$ preserves $\sD(D)$ and $[D,\De^{-1/2}] : \sD(D) \to X$ extends to a bounded adjointable operator $d(\De^{-1/2}) : X \to X$, see \cite[Proposition 3.12]{BlCu:DNS}. It therefore follows by \cite[Proposition 6.7]{Kaa:DAH} and \cite[Example 2 and 3]{Wor:UAQ} that $\De^{-1/2} D \De^{-1/2} : \sD(D) \to X$ is an unbounded selfadjoint and regular operator. By the identity in \eqref{eq:modures} we also obtain that $(i + \De^{-1/2} D \De^{-1/2})^{-1} : X \to X$ is a compact operator. It therefore suffices to show that $\pi(a)$ preserves $\sD(D) = \sD(\De^{-1/2} D \De^{-1/2})$ and that $[\De^{-1/2} D \De^{-1/2}, \pi(a)] : \sD(D) \to X$ extends to a bounded adjointable operator for all $a \in \C A$. Since both $\pi(a) \De$ and $\De^{-1}$ preserves $\sD(D)$ we conclude that $\pi(a)$ preserves $\sD(D)$ for all $a \in \C A$. For $a \in \C A$ and $\xi \in \sD(D)$, a direct computation then shows that
\[
\begin{split}
& \big( \De^{-1/2} D \De^{-1/2} \pi(a) - \pi(a) \De^{-1/2} D \De^{-1/2} \big)(\xi) \\
& \q = \big( \De^{-1/2} \rho_\De(a,0) \De^{-1/2} + \De^{-1/2} d(\De^{-1/2}) \pi(a) + \pi(a)  d(\De^{-1/2}) \De^{-1/2}\big)(\xi)
\end{split}
\]
This proves the proposition.
\end{proof}

We emphasize that the above procedure of transforming an unbounded modular cycle $(X,D,\De)$ into an unbounded Kasparov module does \emph{not} work when $\De : X \to X$ has zero in the spectrum (thus in particular when the $C^*$-algebra $A$ is non-unital). Indeed, the unbounded operator $\De^{-1/2} D \De^{-1/2}$ (when one takes proper care of how to define it) may even fail to have a selfadjoint extension, see \cite[Example 5.1]{Kaa:DAH}.
\medskip

We end this section by defining direct sums of unbounded modular cycles:

\begin{dfn}
Let $\sD_1 := (X_1, D_1, \De_1)$ and $\sD_2 := (X_2, D_2, \De_2)$ be two unbounded modular cycles of the same parity and both from $\C A$ to $B$ (the grading operators are denoted by $\ga_1 : X_1 \to X_1$ and $\ga_2 : X_2 \to X_2$ in the even case). By the \emph{direct sum} of $\sD_1$ and $\sD_2$ we will understand the unbounded modular cycle $\sD_1 \op \sD_2 := (X_1 \op X_2, D_1 \op D_2, \De_1 \op \De_2)$ from $\C A$ to $B$ with grading operator $\ga_1 \op \ga_2 : X_1 \op X_2 \to X_1 \op X_2$ in the even case.
\end{dfn}

\section{Equivalence relations on unbounded modular cycles}\label{s:equrelII}
The construction of the unbounded Kasparov product given in \cite[Section 7]{Kaa:UKM} relies on the choice of a differentiable generating sequence (see Definition \ref{d:cordif}). In order to have a well-defined operation (independent of choices) it is therefore necessary to work with unbounded modular cycles modulo a suitable equivalence relation. We are in this text striving to retain as detailed information about the geometry as possible and we have therefore chosen to work with a modular analogue of bounded perturbations of unbounded Kasparov modules. It is the purpose of this section to introduce this kind of bounded modular perturbations.
\medskip

Let $\C A$ be an operator $*$-algebra satisfying Assumption \ref{a:opealg} and let $B$ be a $\si$-unital $C^*$-algebra.
\medskip

We start by introducing unitary equivalences:

\begin{dfn}\label{d:uniequ}
Two unbounded modular cycles $\sD_1 = (X_1, D_1, \De_1)$ and $\sD_2 = (X_2, D_2, \De_2)$ (both from $\C A$ to $B$, of the same parity, and with grading operators $\ga_1 : X_1 \to X_1$ and $\ga_2 : X_2 \to X_2$ in the even case) are said to be \emph{unitarily equivalent} when there exists a unitary operator $U : X_1 \to X_2$ such that
\[
D_2 = U D_1 U^* \, , \, \, \De_2 = U \De_1 U^* \, , \, \, \pi_2(a) = U \pi_1(a) U^*
\]
for all $a \in A$. In the even case, it is also required that $\ga_2 = U \ga_1 U^*$. When $\sD_1$ and $\sD_2$ are unitarily equivalent we write $\sD_1 \sim_u \sD_2$. We will in this case also apply the notation $\sD_2 = U \sD_1 U^*$.
\end{dfn}


It is clear that the relation $\sim_u$ is an equivalence relation. 

\begin{dfn}\label{d:bouper}
Let $\sD_1 = (X, D_1, \De_1)$ and $\sD_2 = (X, D_2, \De_2)$ be two unbounded modular cycles (both from $\C A$ to $B$, of the same parity, and with the same grading operator $\ga : X \to X$ in the even case). We will say that $\sD_1$ is a \emph{bounded modular perturbation} of $\sD_2$ when the following conditions hold:
\begin{enumerate}
\item The inclusion $(\pi(a) + \la) \De_2(\xi) \in \sD(D_1)$ holds for all $\xi \in \sD(D_2)$ and $(a,\la) \in \wit{\C A}$.
\item There exists a completely bounded linear map
\[
\rho_{\De_1,\De_2} : \wit{\C A} \to \sL(X)
\]
such that
\[
\begin{split}
& \big( \De_1^{1/2} \rho_{\De_1,\De_2}(a,\la) \De_2^{1/2} \big)(\xi) \\
& \q = \big( D_1 (\pi(a) + \la) \De_2 - \De_1( \pi(a) + \la) D_2 \big)(\xi)
\end{split}
\]
for all $\xi \in \sD(D_2)$.
\end{enumerate}
%
\end{dfn}


\begin{lemma}\label{l:bouper}
The relation ``bounded modular perturbation'' is reflexive and symmetric.
\end{lemma}
\begin{proof}
Reflexivity follows from the conditions $(2)$ and $(3)$ of Definition \ref{d:unbkas}.

To prove symmetry we suppose that $\sD_1 = (X_1,D_1,\De_1)$ is a bounded modular perturbation of $\sD_2 = (X_2,D_2,\De_2)$.

We define the completely bounded linear map
\[
\rho_{\De_2,\De_1} : \wit{\C A} \to \sL(X) \q \rho_{\De_2,\De_1}(a,\la) := - \big( \rho_{\De_1,\De_2}(a^*, \ov{\la}) \big)^*
\]

Let now $\xi \in \sD(D_1)$, $\eta \in \sD(D_2)$ and $(a,\la) \in \wit{\C A}$ be given. Suppressing the $*$-homomorphism $\pi : A \to \sL(X)$ we compute that
\[
\begin{split}
\inn{(a + \la)\De_1(\xi), D_2(\eta)} 
& = \inn{\xi, \De_1 (a^* + \ov{\la}) D_2(\eta)} \\
& = \inn{\xi, D_1 (a^* + \ov{\la}) \De_2(\eta)} - \inn{\xi, \De_1^{1/2} \rho_{\De_1,\De_2}(a^*, \ov{\la}) \De_2^{1/2} (\eta)} \\
& = \inn{\De_2 (a + \la) D_1(\xi), \eta} + \inn{\De_2^{1/2} \rho_{\De_2,\De_1}(a,\la) \De_1^{1/2}(\xi),\eta}
\end{split}
\]
Since $D_2 : \sD(D_2) \to X$ is selfadjoint, this shows that $(a + \la)\De_1(\xi) \in \sD(D_2)$ and that
\[
\De_2^{1/2} \rho_{\De_2,\De_1}(a,\la) \De_1^{1/2}(\xi) = D_2 (a + \la) \De_1(\xi) - \De_2 (a + \la) D_1(\xi)
\]
This proves that $\sD_2$ is a bounded modular perturbation of $\sD_1$.
\end{proof}

In order to get some familiarity with the relation ``bounded modular perturbation'' we record the following:

\begin{lemma}
Let $(X,D,1)$ be an unbounded modular cycle from $\C A$ to $B$ and let $G : X \to X$ be a positive invertible bounded operator which we require to be even when $(X,D,1)$ is even. Suppose that $G(\xi) \in \sD(D)$ for all $\xi \in \sD(D)$ and that $[D,G] : \sD(D) \to X$ has a bounded adjointable extension $d(G) : X \to X$. Then the triple $(X,GDG,G^2)$ (with $\sD(GDG) := \sD(D)$) is an unbounded modular cycle from $\C A$ to $B$ of the same parity as $(X,D,1)$ and moreover we have that $(X,D,1)$ and $(X,GDG,G^2)$ are bounded modular perturbations of each other.
\end{lemma}
\begin{proof}
We already saw in Proposition \ref{p:confeq} that $(X,GDG,G^2)$ is an unbounded modular cycle of the same parity as $(X,D,1)$. Moreover, we note that the derivation $d : \C A \to \sL(X)$ (given by the commutator with $D$) is completely bounded since $(X,D,1)$ is an unbounded modular cycle. In particular, we clearly have that condition $(1)$ of Definition \ref{d:bouper} is satisfied. In order to verify condition $(2)$ of Definition \ref{d:bouper} we define the completely bounded map
\[
\rho_{G^2,1} : \wit{\C A} \to \sL(X) \q \rho_{G^2,1}(a,\la) := d(G)\cd ( \pi(a) + \la) + G \cd d(a)
\]
For all $\xi \in \sD(D)$ it then holds that
\[
G \cd \rho_{G^2,1}(a,\la)(\xi) = \big( G D G \cd ( \pi(a) + \la) - G \cd ( \pi(a) + \la) D \big)(\xi)
\]
proving the lemma.
\end{proof}

The next lemma investigates the relation between unitary equivalences and bounded modular perturbations:

\begin{lemma}\label{l:conuni}
Let $Y$ be a (right) Hilbert $C^*$-module over $B$ and let $\sD_1 = (X, D_1,\De_1)$ and $\sD_2 = (X, D_2, \De_2)$ be unbounded modular cycles from $\C A$ to $B$. Suppose that there exists a unitary operator $U : X \to Y$ and that $\sD_1$ is a bounded modular perturbation of $\sD_2$. Then $U \sD_1 U^*$ is a bounded modular perturbation of $U \sD_2 U^*$.
\end{lemma}
\begin{proof}
The first condition of Definition \ref{d:bouper} is clearly satisfied for $U \sD_1 U^*$ and $U \sD_2 U^*$. 

Let $\rho_{\De_1,\De_2} : \wit{\C A} \to \sL(X)$ denote the completely bounded map which implements the bounded modular perturbation from $\sD_1$ to $\sD_2$. The second condition of Definition \ref{d:bouper} then holds with $\rho_{U \De_1 U^*, U \De_2 U^*}(a,\la) := U \rho_{\De_1,\De_2}(a,\la) U^*$ for all $(a,\la) \in \wit{\C A}$. 
\end{proof}

We ignore for the moment whether the relation ``bounded modular perturbation'' is transitive or not. We therefore make the following:

\begin{dfn}
The equivalence relation $\sim_{\T{bmp}}$ on unbounded modular cycles from $\C A$ to $B$ of the same parity is defined by:
\[
\begin{split}
& \Big( \sD_1 = (X_1, D_1, \De_1) \sim_{\T{bmp}} \sD_2 = (X_2, D_2, \De_2) \Big) \\ 
& \q \lrar 
\Big( \T{there exists a unitary operator } U : X_1 \to X_2 \T{ such that } U \sD_1 U^* \T{ and } \sD_2 \\
& \qqq \T{ agree up to a finite number of bounded modular perturbations } \Big)
\end{split}
\]
By a slight abuse of language we will sometimes say that $\sD_1$ and $\sD_2$ are bounded modular perturbations of each other when $\sD_1 \sim_{\T{bmp}} \sD_2$.
\end{dfn}

We are now ready to make a tentative definition of unbounded bivariant $K$-theory.


\begin{dfn}
By the \emph{even (resp. odd) unbounded bivariant $K$-theory} from $\C A$ to $B$ we will understand the collection of even (resp. odd) unbounded modular cycles from $\C A$ to $B$ modulo the equivalence relation $\sim_{\T{bmp}}$. The even (resp. odd) unbounded bivariant $K$-theory from $\C A$ to $B$ is denoted by $UK_0(\C A, B)$ and $UK_1(\C A,B)$ (resp.).
\end{dfn}

\begin{lemma}
The direct sum of even (resp. odd) unbounded modular cycles provides $UK_0(\C A,B)$ and $UK_1(\C A,B)$ with the structure of an abelian monoid. 
\end{lemma}

\begin{remark}
The equivalence relation $\sim_{\T{bmp}}$ is an unbounded analogue of the notion of ``compact perturbation'' and unitary equivalence at the level of bounded Kasparov modules, see \cite[Definition 17.2.4]{Bla:KOA}. We have in this text chosen \emph{not} to work modulo any notion of degenerate unbounded modular cycles. This is due to the significant amount of spectral information that can be contained in a degenerate unbounded modular cycle. It is also worthwhile to notice that whereas the bounded notion of ``compact perturbation'' is insensitive to the growth properties of eigenvalues this is not the case with the more refined notion of perturbations that we are working with in this text (this observation is a consequence of the resolvent identity).
\end{remark}

\section{The unbounded Kasparov product}\label{s:unbkas}
Let $\C A$ and $\C B$ be operator $*$-algebras satisfying the conditions in Assumption \ref{a:opealg} and let $C$ be a $\si$-unital $C^*$-algebra. We recall that all operator $*$-correspondences are assumed to satisfy the conditions in Assumption \ref{a:opecor}.

We are going to construct the \emph{unbounded Kasparov product} of the following two objects:
\begin{enumerate}
\item A compact operator $*$-correspondence $\C X$ from $\C A$ to $\C B$;
\item An unbounded modular cycle $\sD := (Y,D,\Ga)$ from $\C B$ to $C$ with grading operator $\ga : Y \to Y$ in the even case.
\end{enumerate}
This operation is given by an \emph{explicit formula} and produces an unbounded modular cycle from $\C A$ to $C$ of the same parity as $\sD$. 
\medskip

Let us consider the $C^*$-correspondence $\ell^2(\nn,Y)$ from $K_B$ to $C$. We recall that this $C^*$-correspondence is defined as the completion of the algebraic direct sum $\op_{n = 1}^\infty Y$ with respect to the (norm coming from the) inner product $\inn{\sum_{n = 1}^\infty y_n e_{n1}, \sum_{n = 1}^\infty y_n' e_{n1}}_{\ell^2(\nn,Y)} := \sum_{n = 1}^\infty \inn{y_n,y_n'}_Y$. For the definition of the bimodule structure we refer to Subsection \ref{ss:rowcol}. 

We may define the unbounded selfadjoint and regular operator
\[
\T{diag}(D) : \sD\big( \T{diag}(D) \big) \to \ell^2(\nn,Y) \q \T{diag}(D)\big( \sum_{n = 1}^\infty y_n e_{n1} \big) := \sum_{n = 1}^\infty D(y_n) e_{n1}
\]
where the domain is given by
\[
\begin{split}
& \sD\big( \T{diag}(D) \big) \\
& \q := \big\{ \sum_{n = 1}^\infty y_n e_{n1} \in \ell^2(\nn,Y) \mid y_n \in \sD(D) \, \, \forall n \in \nn \, , \, \, \sum_{n = 1}^\infty D(y_n) e_{n1} \in \ell^2(\nn,Y) \big\}
\end{split}
\]
In a similar fashion we obtain the bounded positive operator with dense image $\T{diag}(\Ga) : \ell^2(\nn,Y) \to \ell^2(\nn,Y)$ defined by $\T{diag}(\Ga)\big( \sum_{n = 1}^\infty y_n e_{n1} \big) := \sum_{n = 1}^\infty \Ga(y_n) e_{n1}$.
\medskip

Using Lemma \ref{l:diflinI} we may choose a sequence $\{ \xi_n\}$ of generators for the $C^*$-completion $X$ of the compact operator $*$-correspondence $\C X$ such that
\[
\xi := \sum_{n = 1}^\infty \xi_n e_{1n} \in \ell^2(\nn,\C X)^t
\]
Furthermore, (still by Lemma \ref{l:diflinI}) we obtain that the $C^*$-completion $X$ is differentiable from $\C A$ to $\C B$ with differentiable generating sequence $\{ \xi_n \}$, see Definition \ref{d:cordif}. As in \cite[Lemma 7.1]{Kaa:UKM} we then have a well-defined bounded adjointable operator 
\begin{equation}\label{eq:phidef}
\Phi : X \hot_B Y \to \ell^2(\nn,Y) \q \Phi : x \hot_B y \mapsto \sum_{n = 1}^\infty \inn{\xi_n,x}_X \cd y \cd e_{n1}
\end{equation}
and the image of $\Phi^* : \ell^2(\nn,Y) \to X \hot_B Y$ is norm-dense. 
\medskip

We recall the following constructions from \cite[Section 5 and 7]{Kaa:UKM}:

\begin{dfn}
By the \emph{modular lift} of $D : \sD(D) \to Y$ with respect to $\xi \in \ell^2(\nn,\C X)^t$ we understand the closure $D_\xi$ of the symmetric unbounded operator
\[
(D_\xi)_0 := \Phi^* \T{diag}(D) \Phi : \sD(\T{diag}(D) \Phi) \to X \hot_B Y
\]
where the core for $D_\xi$ is given by
\[
\sD(\T{diag}(D) \Phi) := \big\{ \ze \in X \hot_B Y \mid \Phi(\ze) \in \sD(\T{diag}(D)) \big\}
\]
\end{dfn}

\begin{dfn}
By the \emph{modular lift} of $\Ga : Y \to Y$ with respect to $\xi \in \ell^2(\nn,\C X)^t$ we understand the bounded adjointable operator
\[
\Ga_\xi := \Phi^* \T{diag}(\Ga) \Phi : X \hot_B Y \to X \hot_B Y
\]
\end{dfn}

\begin{dfn}
By the \emph{unbounded Kasparov product} of $\C X$ and $\sD$ (with respect to $\xi \in \ell^2(\nn,\C X)^t$) we understand the triple
\[
\C X \hot_{\C B} \sD := (X \hot_B Y, D_\xi, \Ga_\xi)
\]
with grading operator $1 \ot \ga : X \hot_B Y \to X \hot_B Y$ in the even case.
\end{dfn}

The next result is an immediate consequence of \cite[Theorem 7.1]{Kaa:UKM}:

\begin{thm}
The unbounded Kasparov product $\C X \hot_{\C B} \sD$ is an unbounded modular cycle from $\C A$ to $C$ of the same parity as $\sD$.
\end{thm}

We will now investigate how the class of the unbounded Kasparov product in unbounded bivariant $K$-theory depends on $\C X$ and $\sD$. The next theorem will in particular imply that the class $[\C X \hot_{\C B} \sD] \in UK_*(\C A,C)$ is independent of the choice of generating sequence $\{ \xi_n\}$.

We recall that $\ell^2(\nn,X)^t$ denotes the $C^*$-correspondence from $A$ to $K_B$ defined as the completion of the algebraic direct sum $\op_{n = 1}^\infty X$ with respect to (the norm coming from the) inner product $\inn{\sum_{n = 1}^\infty x_n e_{1n},\sum_{n = 1}^\infty x_n' e_{1n}}_{\ell^2(\nn,X)^t} := \sum_{n,m = 1}^\infty \inn{x_n,x_m'} e_{nm}$. We refer to Subsection \ref{ss:rowcol} for the definition of the bimodule structure.

\begin{thm}
The class $[\C X \hot_{\C B} \sD] \in UK_*(\C A,C)$ only depends on the classes $[\C X] \in M(\C A,\C B)$ and $[\sD] \in UK_*(\C B,C)$. In other words, we have a well-defined map
\[
\hot_{\C B} : M(\C A,\C B) \ti UK_*(\C B,C) \to UK_*(\C A,C)
\]
which is induced by the unbounded Kasparov product. 
\end{thm}
\begin{proof}
Let $\C X'$ be an alternative compact operator $*$-correspondence from $\C A$ to $\C B$ and let $\sD' = (Y',D',\Ga')$ be an alternative unbounded modular cycle from $\C B$ to $C$. We suppose that $[\C X] = [\C X']$ in $M(\C A,\C B)$ and that $[\sD] = [\sD'] \in UK_*(\C B,C)$. We are going to show that $[\C X \hot_{\C B} \sD] = [\C X' \hot_{\C B} \sD']$ in $UK_*(\C A,C)$.

It is not difficult to show that if $\C X$ and $\C X'$ are unitarily equivalent or $\sD$ and $\sD'$ are unitarily equivalent then the unbounded Kasparov products $\C X \hot_{\C B} \sD$ and $\C X' \hot_{\C B} \sD'$ are unitarily equivalent as well. We may thus assume that $\C X$ and $\C X'$ have the same $C^*$-completion $X$ and that the inner product on $X$ induces a completely bounded pairing:
\[
\inn{\cd,\cd}_X : \C X \ti \C X' \to \C B
\]
Moreover, we may assume that $Y = Y'$ and that $\sD$ is a bounded modular perturbation of $\sD'$ via the completely bounded map
\[
\rho_{\Ga,\Ga'} : \wit{\C B} \to \sL(Y)
\]

Let $\{ \eta_n \}$ be an alternative sequence of generators for $X$ such that
\[
\eta := \sum_{n = 1}^\infty \eta_n e_{1n} \in \ell^2(\nn,\C X')^t
\]
and let $\Psi : X \hot_B Y \to \ell^2(\nn,Y)$ denote the associated bounded adjointable operator, defined as in \eqref{eq:phidef}. We let $D'_\eta : \sD(D'_\eta) \to X \hot_B Y$ and $\Ga'_\eta : X \hot_B Y \to X \hot_B Y$ denote the corresponding modular lifts of $D' : \sD(D') \to Y$ and $\Ga' : Y \to Y$. 

We will show that the unbounded modular cycle $(X \hot_B Y, D_\xi , \Ga_\xi)$ is a bounded modular perturbation of the unbounded modular cycle $(X \hot_B Y, D'_\eta , \Ga'_\eta)$. 

We first recall from \cite[Proposition 3.5]{Kaa:UKM} that the triples
\[
\begin{split}
\T{diag}(\sD) & := \big(\ell^2(\nn,Y), \T{diag}(D), \T{diag}(\Ga)\big) \q \T{and} \\
\T{diag}(\sD') & := \big(\ell^2(\nn,Y), \T{diag}(D'), \T{diag}(\Ga')\big)
\end{split}
\]
are unbounded modular cycles from $K_{\C B}$ to $C$ both having grading operator $\T{diag}(\ga) : \ell^2(\nn,Y) \to \ell^2(\nn,Y)$ in the even case. We denote the associated $*$-homomorphism by $\pi_B : K_B \to \sL( \ell^2(\nn,Y))$. 

We then remark that $\T{diag}(\sD)$ is a bounded modular perturbation of $\T{diag}(\sD')$. Indeed, the relevant completely bounded map is defined by
\[
\begin{split}
& \rho_{\T{diag}(\Ga), \T{diag}(\Ga')} : \wit{ K_{\C B} } \to \sL( \ell^2(\nn,Y)) \\
& \rho_{\T{diag}(\Ga), \T{diag}(\Ga')}\big( \sum_{n,m = 1}^\infty b_{nm} e_{nm} \, , \, \la \big) 
:= \sum_{n,m = 1}^\infty \rho_{\Ga, \Ga'}( b_{nm}, 0 ) e_{nm} + \sum_{n = 1}^\infty \rho_{\Ga,\Ga'}(0,\la) e_{nn}
\end{split}
\]
for all $\big( \sum_{n,m = 1}^\infty b_{nm} e_{nm} \, , \, \la \big) \in \wit{ K_{\C B} }$.

Let now $(a,\la) \in \wit{\C A}$ be given. A direct computation shows that we have the identity 
\[
\Phi  \big( (\pi(a) + \la) \hot 1\big) \Psi^* = \pi_B\big( \inn{\xi, (a + \la) \cd \eta}_{\ell^2(\nn,X)^t} \big)
\]
of bounded adjointable operators on $\ell^2(\nn,Y)$. Furthermore, since the pairing $\inn{\cd,\cd}_X : \C X \ti \C X' \to \C B$ is completely bounded, we obtain that $\inn{\xi,(a + \la) \cd \eta}_{\ell^2(\nn,X)^t} \in K_{\C B}$.

Let now $\ze \in X \hot_B Y$ with $\Psi(\ze) \in \sD( \T{diag}(D'))$ be given. The identity
\[
\begin{split}
\Phi \big( (\pi_A(a) + \la) \hot 1 \big)\Ga'_\eta (\ze) 
= \pi_B\big( \inn{\xi,(a + \la)\cd \eta}_{\ell^2(\nn,X)^t} \big) \cd \T{diag}(\Ga') \Psi  ( \ze)
\end{split}
\]
shows that $\big( (\pi_A(a) + \la) \hot 1\big)\Ga'_\eta(\ze) \in \sD(D_\xi)$. Thus, condition $(1)$ of Definition \ref{d:bouper} holds for all elements in the core $\sD\big( \T{diag}(D') \Psi \big) \su X \hot_B Y$ for the modular lift $D'_\eta : \sD( D'_\eta) \to X \hot_B Y$.

To continue, we remark that Lemma \ref{l:imahal} implies that
\[
\begin{split}
& S_\xi := \Ga_\xi^{-1/2} \Phi^* \T{diag}(\Ga)^{1/2} : \ell^2(\nn,Y) \to X \hot_B Y \q  \T{and} \\
& S'_\eta := (\Ga'_\eta)^{-1/2} \Psi^* \T{diag}(\Ga')^{1/2}  : \ell^2(\nn,Y) \to X \hot_B Y
\end{split}
\]
are well-defined bounded adjointable operators. We may thus introduce the completely bounded map
\[
\begin{split}
& \rho_{\Ga_\xi, \Ga'_\eta} : \wit{\C A} \to \sL(X \hot_B Y) \\
& \rho_{\Ga_\xi, \Ga'_\eta}(a,\la) 
:=  S_\xi \cd \rho_{\T{diag}(\Ga), \T{diag}(\Ga')}\big(  \inn{\xi,(a + \la)\cd \eta}_{\ell^2(\nn,X)^t},0 \big) \cd (S'_\eta)^* 
\q (a,\la) \in \wit{\C A}
\end{split}
\]
For $\ze \in X \hot_B Y$ with $\Psi(\ze) \in \sD( \T{diag}(D'))$ and $(a,\la) \in \wit{\C A}$, we then have that
\[
\begin{split}
& \Ga_\xi^{1/2} \cd \rho_{\Ga_\xi, \Ga_\eta'}(a,\la) \cd (\Ga_\eta')^{1/2}(\ze) \\
& \q = \Phi^* \T{diag}(\Ga)^{1/2} \cd \rho_{\T{diag}(\Ga),\T{diag}(\Ga')}\big(  \inn{\xi,(a + \la)\cd \eta}_{\ell^2(\nn,X)^t},0 \big) \cd \T{diag}(\Ga')^{1/2} \Psi(\ze) \\
& \q = \Phi^* \Big( \T{diag}(D) \pi_B\big(\inn{\xi,(a + \la) \cd \eta}_{\ell^2(\nn,X)^t}\big) \T{diag}(\Ga') \\
& \qqq - \T{diag}(\Ga) \pi_B\big(\inn{\xi,(a + \la) \cd \eta}_{\ell^2(\nn,X)^t}\big) \T{diag}(D') \Big) \Psi(\ze) \\
& \q = D_\xi \big( (\pi_A(a) + \la) \hot 1 \big) \Ga'_\eta(\ze)
- \Ga_\xi  \big( (\pi_A(a) + \la) \hot 1 \big) D'_\eta (\ze)
\end{split}
\]
This shows that condition $(2)$ of Definition \ref{d:bouper} holds on the core $\sD\big( \T{diag}(D') \Psi \big) \su X \hot_B Y$ for the modular lift $D'_\eta : \sD( D'_\eta) \to X \hot_B Y$. But this means that $(X \hot_B Y, D_\xi, \Ga_\xi)$ is a bounded modular perturbation of $(X \hot_B Y, D'_\eta, \Ga'_\eta)$ and we have proved the theorem.
\end{proof}

%

\subsection{Associativity and bilinearity}
Let $\C D$, $\C A$ and $\C B$ be operator $*$-algebras satisfying Assumption \ref{a:opealg} and let $C$ be a $\si$-unital $C^*$-algebra. As usually all operator $*$-correspondences satisfy Assumption \ref{a:opecor}.

We are now going to investigate the algebraic properties of the unbounded Kasparov product. We emphasize that the following proof of the associativity of the unbounded Kasparov product relies very much on the properties of the Haagerup tensor product of operator spaces.

\begin{thm}
The unbounded Kasparov product $\hot_{\C B} : M(\C A,\C B) \ti UK_*(\C B,C) \to UK_*(\C A,C)$ is bilinear and associative.
\end{thm}
\begin{proof}
Let $\C X_1$ and $\C X_2$ be two compact operator $*$-correspondences from $\C A$ to $\C B$ (with $C^*$-completions $X_1$ and $X_2$) and let $\sD = (Y,D,\De)$ be an unbounded modular cycle from $\C B$ to $C$. Choose a sequence of generators $\{ \xi_n \}$ for $X_1$ and a sequence of generators $\{ \eta_n\}$ for $X_2$ such that
\[
\xi = \sum_{n = 1}^\infty \xi_n e_{1n} \in \ell^2(\nn,\C X_1)^t \q \T{and} \q \eta = \sum_{n = 1}^\infty \eta_n e_{1n} \in \ell^2(\nn,\C X_2)^t 
\]
For each $n \in \nn$ we define the element 
\[
\ze_n := \fork{ccc}{ (\xi_k,0) & \T{for} & n = 2k - 1 \\ (0, \eta_k) & \T{for} & n = 2k}
\]
It follows that $\{ \ze_n\}$ is a sequence of generators for the direct sum $X_1 \op X_2$ and that
\[
\ze = \sum_{n = 1}^\infty \ze_n e_{1n} \in \ell^2(\nn, \C X_1 \op \C X_2)^t
\]
It can then be verified that the unbounded modular cycles 
\[
(X_1 \hot_B Y, D_\xi, \De_\xi) \op (X_2 \hot_B Y, D_\eta, \De_\eta) \q \T{and} \q
\big( (X_1 \op X_2) \hot_B Y, D_\ze, \De_\ze \big)
\]
are unitarily equivalent, proving bilinearity in the first variable.
%
%
%

Let $\C X$ be a compact operator $*$-correspondence from $\C A$ to $\C B$ (with $C^*$-completion $X$) and let $\sD_1 = (Y_1, D_1, \De_1)$ and $\sD_2 = (Y_2, D_2, \De_2)$ be unbounded modular cycles from $\C B$ to $C$. We choose a generating sequence $\{ \xi_n \}$ for $X$ such that
\[
\sum_{n = 1}^\infty \xi_n \cd e_{1n} \in \ell^2(\nn,\C X)^t
\]
It can then be verified that the unbounded modular cycles
\[
\begin{split}
& \big(X \hot_B Y_1, (D_1)_\xi, (\De_1)_\xi \big) \op \big(X \hot_B Y_2, (D_2)_\xi, (\De_2)_\xi \big) \q \T{and} \\
& \big(X \hot_B (Y_1 \op Y_2), (D_1 \op D_2)_\xi, (\De_1 \op \De_2)_\xi\big)
\end{split}
\]
are unitarily equivalent, proving bilinearity in the second variable.
%

Let $\C X$ and $\C Y$ be two compact operator $*$-correspondences from $\C D$ to $\C A$ and from $\C A$ to $\C B$, respectively, (with $C^*$-completions $X$ and $Y$). Furthermore, we let $\sD = (Z, D ,\De)$ be an unbounded modular cycle from $\C B$ to $C$. We choose generating sequences $\{\xi_n\}$ and $\{ \eta_m\}$ for $X$ and $Y$, respectively, such that
\[
\sum_{n = 1}^\infty \xi_n \cd e_{1n} \in \ell^2(\nn,\C X)^t \q \T{and} \q
\sum_{m = 1}^\infty \eta_m \cd e_{1m} \in \ell^2(\nn,\C Y)^t
\]
Let us fix an isomorphism of sets $\al : \nn \ti \nn \to \nn$. We then have the generating sequence $\{ \ze_k\}$ for $X \hot_B Y$ defined by
\[
\ze_{\al(n,m)} := \xi_n \hot_B \eta_m
\]
We claim that
\[
\ze = \sum_{k = 1}^\infty \ze_k \cd e_{1k} \in \ell^2(\nn,\C X \hot_{\C B} \C Y)^t
\]
Indeed, this last convergence result can be verified in the following way: For each $M \geq N \geq 1$, we define the rows
\[
\begin{split}
& \xi^{N,M} := \sum_{n = 1}^{M - N + 1} \xi_{N - 1 + n} \cd e_{1n} \in M_{1,M - N + 1}(\C X) \q \T{and} \\
& \eta^{N,M} := \sum_{n = 1}^{M - N + 1} \eta_{N - 1 + n} \cd e_{1n} \in M_{1,M - N + 1}(\C Y)
\end{split}
\]
For each $M \geq 2N \geq 1$ we then have the estimates
\[
\begin{split}
& \| \xi^{1,2N} \ot ( \eta^{2N,M} \op \eta^{2N - 1, M - 1} \op \ldots \op \eta^{1, M - 2N + 1} ) \|_{\C X \wot \C Y} \\
& \q \leq \| \xi^{1,N} \ot ( \eta^{2N,M} \op \eta^{2N - 1, M - 1} \op \ldots \op \eta^{N + 1, M - N + 1} ) \|_{\C X \wot \C Y} \\
& \qqq + \| \xi^{N + 1, 2N} \ot ( \eta^{N,M-N} \op \eta^{N - 1, M - N - 1} \op \ldots \op \eta^{1, M - 2N + 1} ) \|_{\C X \wot \C Y} \\
& \q \leq \| \xi \|_{\ell^2(\nn,\C X)^t} \cd \| \eta^{N + 1, M} \|_{\C Y}
+ \| \xi^{N + 1,2N} \|_{\C X} \cd \| \eta\|_{\ell^2(\nn,\C Y)^t}
\end{split}
\]
using the matrix norms for the Haagerup tensor product $\C X \wot \C Y$. Since $\xi \in \ell^2(\nn,\C X)^t$ and $\eta \in \ell^2(\nn,\C Y)^t$ by construction, the above estimates imply that $\ze \in \ell^2(\nn,\C X \hot_{\C B} \C Y)^t$. It can then be verified that the unbounded modular cycles
\[
\big( X \hot_B (Y \hot_C Z), (D_\eta)_\xi, (\De_\eta)_\xi \big) \q \T{and} \q
\big( (X \hot_B Y) \hot_C Z, D_\ze, \De_\ze \big)
\]
are unitarily equivalent, proving the associativity of the unbounded Kasparov product.
%
%
\end{proof}


\section{The Baaj-Julg bounded transform}\label{s:baajul}
In this section we will investigate how the Morita monoid $M(\C A,\C B)$ and unbounded bivariant $K$-theory $UK_*(\C A,B)$ relate to Kasparov's analytic $KK$-theory, $KK_*(A,B)$. The link between these theories is provided by the Baaj-Julg bounded transform, \cite{BaJu:TBK}, and our main result is that this operation is compatible with the equivalence relations on compact operator $*$-correspondences and unbounded modular cycles. 
\medskip

The case of the Morita monoid is fairly simple and we state the results without proofs: 

\begin{dfn}\label{d:boutramod}
Let $\C A$ and $\C B$ be operator $*$-algebras satisfying Assumption \ref{a:opealg}. The \emph{bounded transform} of a compact operator $*$-correspondence $\C X$ from $\C A$ to $\C B$ is the even bounded Kasparov module $(X,0)$, where $X$ denotes the $C^*$-completion of $\C X$ and the grading operator is given by $\ga = 1_X : X \to X$.
\end{dfn}

\begin{lemma}
The bounded transform $\C X \mapsto (X,0)$ induces a homomorphism of abelian monoids $F : M(\C A,\C B) \to KK_0(A,B)$.
\end{lemma}

The case of unbounded modular cycles is far more involved. In fact, treating the bounded transform of an unbounded modular cycle $(X,D,\De)$ is even substantially more difficult than treating the bounded transform of an unbounded Kasparov module. The difficulties in the modular setting mainly arise for two reasons:
\begin{enumerate}
\item The modular operator $\De : X \to X$ will typically have an \emph{unbounded} inverse $\De^{-1} : \T{Im}(\De) \to X$.
\item We do not impose any kind of Lipschitz regularity condition. Thus, we do \emph{not} assume that twisted commutators of the form $|D| a \De - \De a |D| : \sD(D) \to X$ extend to bounded operators on $X$.  
\end{enumerate}

Let us fix an operator $*$-algebra $\C A$ satisfying Assumption \ref{a:opealg} and a $\si$-unital $C^*$-algebra $B$. Recall that $A$ denotes the $C^*$-completion of $\C A$.

\begin{dfn}\label{d:boutraunb}
By the \emph{bounded transform} of an unbounded modular cycle $(X, D,\De)$ from $\C A$ to $B$ (with grading $\ga : X \to X$ in the even case) we will understand the pair $(X, D(1 + D^2)^{-1/2})$ (still with grading $\ga : X \to X$ in the even case). 
\end{dfn}

The next result was proved in \cite[Theorem 9.1]{Kaa:UKM}:

\begin{thm}\label{t:boukas}
The bounded transform $(X,D(1 + D^2)^{-1/2})$ of an unbounded modular cycle $(X,D,\De)$ is a bounded Kasparov module from $A$ to $B$ of the same parity as $(X,D,\De)$. In particular, there is a well-defined element $[X,D(1 + D^2)^{-1/2}] \in KK_*(A,B)$ in the $KK$-group of the pair $(A,B)$ of $\si$-unital $C^*$-algebras. 
\end{thm}

As mentioned earlier, the main result of this section is that the bounded transforms of equivalent unbounded modular cycles define the same class in $KK$-theory:

\begin{thm}\label{t:welbou}
The assignment $(X, D,\De) \mapsto (X, D(1 + D^2)^{-1/2})$ induces an even homomorphism of abelian monoids $F : UK_*(\C A,B) \to KK_*(A,B)$. 
\end{thm}
\begin{proof}
It is clear that the map $(X,D,\De) \mapsto (X, D(1 + D^2)^{-1/2})$ respects direct sums and unitary equivalences. It therefore suffices to show that the bounded Kasparov modules $(X,D_1(1 + D_1^2)^{-1/2})$ and $(X,D_2(1 + D_2^2)^{-1/2})$ are compact perturbations of each other whenever the two unbounded modular cycles $(X,D_1,\De)$ and $(X,D_2,\Ga)$ are bounded modular perturbations of each other. The proof of this result will occupy the remainder of this section, see Proposition \ref{p:compermod}.
\end{proof}

\subsection{Bounded modular versus compact perturbations}
Let $\C A$ be an operator $*$-algebra satisfying Assumption \ref{a:opealg} and let $B$ be a $\si$-unital $C^*$-algebra.

We fix two unbounded modular cycles $\sD_1 = (X,D_1,\De)$ and $\sD_2 = (X,D_2,\Ga)$ both from $\C A$ to $B$ of the same parity and with the same grading operator $\ga : X \to X$ in the even case. We will assume that $\sD_1$ is a bounded modular perturbation of $\sD_2$. Thus, by Definition \ref{d:bouper}, there exists a completely bounded map
\[
\rho_{\De,\Ga} : \wit{\C A} \to \sL(X)
\]
such that
\[
\big( \De^{1/2} \rho_{\De,\Ga}(a,\mu) \Ga^{1/2} \big)(\xi) = \big( D_1 (a + \mu) \Ga - \De (a + \mu) D_2 \big)(\xi)
\]
for all $(a,\mu) \in \wit{\C A}$ and all $\xi \in \sD(D_2)$. Remark also that the condition $(a + \mu) \Ga (\xi) \in \sD(D_1)$ is part of our assumptions. We apply the notation
\[
F_{D_1} := D_1(1 + D_1^2)^{-1/2} \q  \T{and} \q F_{D_2} := D_2(1 + D_2^2)^{-1/2}
\]
for the bounded transforms of $D_1 : \sD(D_1) \to X$ and $D_2 : \sD(D_2) \to X$. In order to prove our main Theorem \ref{t:welbou} we will show that
\[
a (F_{D_1} - F_{D_2}) \in \sK(X) \q \T{ for all } a \in A 
\]
thus that the bounded Kasparov modules $(X,F_{D_1})$ and $(X,F_{D_2})$ are compact perturbations of each other.

We start by replacing the bounded transforms by modular transform type operators (compare with \cite[Section 8]{Kaa:UKM}). To this end, we let $r \in (\| \De \|_\infty^2 + \| \Ga \|_\infty^2, \infty)$ be fixed and introduce the bounded adjointable operators
\[
S_\la := (\la \De^2/r + 1 + D_1^2)^{-1} \q \T{and} \q  T_\la := (\la \Ga^2/r + 1 + D_2^2)^{-1} 
\]
for all $\la \geq 0$. We recall from \cite[Section 8]{Kaa:UKM} that the \emph{modular transform} of $(X,D_1,\De)$ is defined as the unbounded operator:
\[
G_{D_1,\De} : \De\big( \sD(D_1) \big) \to X \q 
G_{D_1,\De}( \De \xi) := \frac{1}{\pi} \int_0^\infty (\la r)^{-1/2} \De S_\la D_1(\De \xi) \, d\la
\]
for all $\xi \in \sD(D_1)$. We remark that the above integral converges absolutely since \cite[Lemma 11.3]{Kaa:UKM} implies that
\[
\begin{split}
\| \De S_\la D_1(\De \xi) \|_X 
& \leq \| \De S_\la \De \|_\infty \cd \| D_1 \xi \|_X
+ \| \De S_\la \De^{1/2} \|_\infty \cd \| \rho_{\De}(1) \De^{1/2} \xi \|_X \\
& = O\big( (1 + \la)^{-3/4} \big) \q \T{as } \la \to \infty \T{ and as } \la \to 0
\end{split}
\]
where $\rho_\De : \wit{\C A} \to \sL(X)$ is the completely bounded map associated to the unbounded modular cycle $(X,D_1,\De)$. A similar definition applies to the unbounded modular cycle $(X,D_2,\Ga)$. Recall in this respect that $\rho_\Ga : \wit{\C A} \to \sL(X)$ denotes the corresponding completely bounded map.

The modular transform is a useful analytic tool for handling twisted commutators instead of straight commutators (corresponding to $\De \neq 1$ and $\De = 1$, respectively). The modular transform is obtained from the bounded transform by applying the \emph{formal} change of variables $\la \to \la \De^2/r$ in the integral formula for the square root: 
\[
(1 + D_1^2)^{-1/2} = \frac{1}{\pi} \int_0^\infty \la^{-1/2} (\la + 1 + D_1^2)^{-1} \, d\la
\]
\medskip

The following lemma illustrates the main algebraic reason for working with the modular transform instead of the usual bounded transform. Indeed, the operator norm of the right hand side of the identity here below will (at least after multiplication with modular operators) behave like $(1 + \la)^{-5/4}$. The operator norm of the left hand side only behaves like $(1 + \la)^{-1}$ (after multiplication with modular operators). 

\begin{lemma}\label{l:funalg}
We have the identity
\[
\begin{split}
& a \Ga^2 T_\la - S_\la \De^2 a \\ 
& \q = S_\la(a \Ga^2 - \De^2 a) T_\la \\
& \qqq + S_\la \big( \De^{1/2} \rho_\De(1) \De^{1/2} a \Ga + \De^{3/2} \rho_{\De,\Ga}(a) \Ga^{1/2} \big) D_2 T_\la \\ 
& \qqq \qq +  ( D_1 S_\la )^* \big( \De^{1/2} \rho_{\De,\Ga}(a) \Ga^{3/2} + \De a \Ga^{1/2} \rho_\Ga(1) \Ga^{1/2} \big) T_\la
\end{split}
\]
for all $a \in \C A$. 
\end{lemma}
\begin{proof}
By the properties of the completely bounded maps $\rho_{\De}, \rho_\Ga$ and $\rho_{\De,\Ga} : \wit{\C A} \to \sL(X)$ we have that
\[
\begin{split}
& S_\la \big( \De^{1/2} \rho_\De(1) \De^{1/2} a \Ga + \De^{3/2} \rho_{\De,\Ga}(a) \Ga^{1/2} \big) D_2 T_\la \\ 
& \qqq +  ( D_1 S_\la )^* \big( \De^{1/2} \rho_{\De,\Ga}(a) \Ga^{3/2} + \De a \Ga^{1/2} \rho_\Ga(1) \Ga^{1/2} \big) T_\la \\
& \q = S_\la \big( (D_1 \De - \De D_1 ) a \Ga + \De (D_1 a \Ga - \De a D_2)  \big) D_2 T_\la \\ 
& \qqq +  ( D_1 S_\la )^* \big( (D_1 a \Ga - \De a D_2) \Ga + \De a (D_2 \Ga - \Ga D_2) \big) T_\la \\
& \q = ( D_1 S_\la )^* D_1 a \Ga^2 T_\la - S_\la \De^2 a D_2^2 T_\la \\
& \q = ( D_1 S_\la )^* D_1 a \Ga^2 T_\la - S_\la \De^2 a D_2^2 T_\la
+ S_\la (\De^2/r) a \Ga^2 T_\la - S_\la \De^2 a (\Ga^2/r) T_\la \\
& \q = a \Ga^2 T_\la - S_\la \De^2 a
+ S_\la (\De^2 a - a \Ga^2) T_\la
\end{split}
\]
This proves the lemma.
\end{proof}

The next two propositions relate the bounded transforms $F_{D_2}$ and $F_{D_1}$ to modular transforms. Thus, up to a compact perturbation we may replace the (rescaled) resolvents $\la^{-1/2} (\la + 1 + D_2^2)^{-1}$ and $\la^{-1/2} (\la + 1 + D_1^2)^{-1}$ (in the integral formulae for the bounded transforms) by bounded adjointable operators of the form $(\la r)^{-1/2} \Ga \cd T_\la$ and $(\la r)^{-1/2} \De \cd S_\la$, respectively.


\begin{prop}\label{p:modde1}
For each $a \in \C A$, there exists a compact operator $K_a : X \to X$ such that
\[
\De^5 a F_{D_2} \Ga^5(\xi) - \frac{1}{\pi} \int_0^\infty (\la r)^{-1/2} \De^5 a \Ga^2 T_\la \Ga^4 \, d\la \cd D_2(\xi)
= K_a(\xi)
\]
for all $\xi \in \sD(D_2)$.
\end{prop}
\begin{proof}
Let $a \in \C A$. By \cite[Lemma 9.3]{Kaa:UKM} there exists a compact operator $K_{a^*} : X \to X$ such that
\[
\Ga^5 F_{D_2} a^*(\eta) - \frac{1}{\pi} D_2 \cd \int_0^\infty (\la r)^{-1/2} \Ga^4 T_\la \Ga^2 \, d\la \cd a^*(\eta)
= K_{a^*}(\eta)
\]
for all $\eta \in \Ga\big( \sD(D_2) \big)$. By taking adjoints we thus obtain that
\[
a F_{D_2} \Ga^5(\xi) - \frac{1}{\pi} a \cd \int_0^\infty (\la r)^{-1/2} \Ga^2 T_\la \Ga^4 \, d\la \cd D_2(\xi)
= (K_{a^*})^*(\xi)
\]
for all $\xi \in \sD(D_2)$. This proves the proposition with $K_a := (K_{a^*})^* : X \to X$.
\end{proof}

In order to obtain the analogue of the above proposition for the bounded transform $F_{D_1}$ we introduce the bounded adjointable operators
\[
R_\la := (\la + 1 + D_1^2)^{-1} \q \T{and} \q Y_\la := \la \cd R_\la \cd (1 - \De^2/r)
\]
for all $\la \geq 0$. We remark that $\| Y_\la \|_\infty < 1$ and that \cite[Lemma 8.2]{Kaa:UKM} states that
\begin{equation}\label{eq:eswyer}
S_\la = (1 - Y_\la)^{-1} R_\la
\end{equation}

\begin{prop}\label{p:modde2}
Let $a \in \C A$ be fixed. There exists a compact operator $L_a : X \to X$ such that
\[
\De^5 a F_{D_1} \Ga^5(\xi) - \frac{1}{\pi} \int_0^\infty (\la r)^{-1/2} \De^3 (1 - Y_\la)^{-1} \De^3 R_\la \, d\la \cd D_1 a \Ga^5 (\xi)
= L_a(\xi)
\]
for all $\xi \in \sD(D_2)$.
\end{prop}
\begin{proof}
%
By \cite[Theorem 8.1]{Kaa:UKM} there exists a $B$-linear bounded operator $Q : X \to X$ such that
\[
\begin{split}
& \De^5 F_{D_1} \cd (1 + D_1^2)^{1/4}(\eta) - \frac{1}{\pi} \int_0^\infty (\la r)^{-1/2} \De^3 (1 - Y_\la)^{-1} \De^3 R_\la \, d\la \cd D_1 (1 + D_1^2)^{1/4}(\eta) \\ 
& \q = Q(\eta)
\end{split}
\]
for all $\eta \in \sD( |D_1|^{3/2})$. We thus have that
\[
\begin{split}
& \De^5 F_{D_1} a \Ga^5(\xi) - \frac{1}{\pi} \int_0^\infty (\la r)^{-1/2} \De^3 (1 - Y_\la)^{-1} \De^3 R_\la \, d\la \cd D_1 a \Ga^5(\xi) \\
& \q = Q (1 + D_1^2)^{-1/4} a \Ga^5(\xi) 
\end{split}
\]
for all $\xi \in \sD(D_2)$. Since $(1 + D_1^2)^{-1/4} a : X \to X$ is compact and since $Q : X \to X$ is $B$-linear and bounded we conclude that $Q(1 + D_1^2)^{-1/4} a \Ga^5 : X \to X$ is a compact operator. Combining this result with Theorem \ref{t:boukas}, which tells us that $(X,F_{D_1})$ is a bounded Kasparov module, we obtain the result of the present proposition.
\end{proof}

The idea is now to compare the integrands in Proposition \ref{p:modde1} and Proposition \ref{p:modde2}. Indeed, the following proposition will allow us to conclude that the difference $\De^5 a (F_{D_1} - F_{D_2})\Ga^5 : X \to X$ is a compact operator for all $a \in \C A$.

\begin{prop}\label{p:emmcom}
Let $a \in \C A$ be fixed. For each $\la \geq 0$ there exists a compact operator $M_a(\la) : X \to X$ such that
\[
\big( \De^5 a \Ga^2 T_\la \Ga^4 D_2 - \De^3 (1 - Y_\la)^{-1} \De^3 R_\la D_1 a \Ga^5\big)(\xi) = M_a(\la)(\xi)
\]
for all $\xi \in \sD(D_2)$. Moreover we have that the map $M_a : [0,\infty) \to \sK(X)$ is continuous in operator norm and satisfies the estimate:
\begin{equation}\label{eq:mmmest}
\| M_a(\la) \|_\infty = O\big( (1 + \la)^{-3/4}\big) \q \T{as } \la \to \infty \T{ and as } \la \to 0
\end{equation}
\end{prop}
\begin{proof}
Both of the unbounded operators $\De^5 a \Ga^2 T_\la \Ga^4 D_2 : \sD(D_2) \to X$ and $\De^3 (1 - Y_\la)^{-1} \De^3 R_\la D_1 a \Ga^5 : \sD(D_2) \to X$ have compact extensions to $X$ and it is not hard to see that these compact extensions depend continuously on the parameter $\la \geq 0$. The heart of the matter is therefore to prove the estimate in \eqref{eq:mmmest}. But this estimate will be a consequence of Lemma \ref{l:funestI} and Lemma \ref{l:funest} here below.
\end{proof}

For the convenience of the reader we recall from \cite[Lemma 11.3]{Kaa:UKM} that
\begin{equation}\label{eq:lemest}
\| \De S_\la^{1/2} \|_\infty \, \, , \, \, \, \| \Ga T_\la^{1/2} \|_\infty \leq \frac{\sqrt{2r}}{\sqrt{1 + \la}} \q \T{ for all } \la \geq 0 
\end{equation}

\begin{lemma}\label{l:funestI}
There exists a constant $C > 0$ such that
\[
\big\| \De (1 - Y_\la)^{-1} \De^3 R_\la D_1(\xi) - \De^3 S_\la \De D_1(\xi) \big\|_X \leq C \cd (1 + \la)^{-3/4} \cd \| \xi \|_X
\]
for all $\xi \in \sD(D_1)$ and all $\la \geq 0$.
\end{lemma}
\begin{proof}
Let $\xi \in \sD(D_1)$. We start by recording the algebraic identities (relying on \eqref{eq:eswyer}):
\[
\begin{split}
& \De (1 - Y_\la)^{-1} \De^3 R_\la \cd D_1(\xi) - \De^3 S_\la \De \cd D_1(\xi) \\
& \q = \De(1 - Y_\la)^{-1}[\De^3,R_\la] \cd D_1(\xi) - \De[\De^2,S_\la] \De \cd D_1(\xi) \\
& \q = \De(1 - Y_\la)^{-1} D_1 R_\la [D_1,\De^3] R_\la D_1(\xi)
+ \De(1 - Y_\la)^{-1} R_\la [D_1,\De^3] D_1  R_\la D_1(\xi) \\
& \qqq
- \De (D_1 S_\la)^* [D_1,\De^2] S_\la \De \cd D_1(\xi)
- \De S_\la [D_1,\De^2] D_1 S_\la \De \cd D_1(\xi) \\
& \q = \De (D_1 S_\la)^* [D_1,\De^3] R_\la D_1(\xi)
+ \De S_\la [D_1,\De^3] D_1 R_\la D_1(\xi) \\
& \qqq - \De (D_1 S_\la)^* [D_1,\De^2] S_\la \De \cd D_1(\xi)
- \De S_\la [D_1,\De^2] D_1 S_\la \De \cd D_1(\xi)
\end{split}
\]
The result of the lemma then follows by \eqref{eq:lemest}.
\end{proof}
%
%

\begin{lemma}\label{l:funest}
Let $a \in \C A$ be fixed. There exists a constant $C > 0$ such that
\[
\big\| \De a \Ga^2 T_\la \Ga^4 D_2(\xi) - \De S_\la \De D_1 a \Ga^5(\xi) \big\|_X \leq C \cd (1 + \la)^{-3/4} \cd \| \xi \|_X
\]
for all $\xi \in \sD(D_2)$ and all $\la \geq 0$.
\end{lemma}
\begin{proof}
%
Let $\xi \in \sD(D_2)$ and notice that
\[
\begin{split}
& \De a \Ga^2 T_\la \Ga^4 D_2(\xi) - \De S_\la \De D_1 a \Ga^5(\xi)  \\
& \q = \De ( a \Ga^2 T_\la - S_\la \De^2 a) \Ga^4 D_2(\xi) 
+ \De S_\la( \De^2 a \Ga^4 D_2 - \De D_1 a \Ga^5)(\xi)  \\
& \q = \De ( a \Ga^2 T_\la - S_\la \De^2 a) \Ga^4 D_2(\xi) 
+ \De S_\la( \De^2 a \Ga^4 D_2 - \De^2 a D_2 \Ga^4)(\xi) \\
& \qqq - \De S_\la \De^{3/2} \rho_{\De,\Ga}(a) \Ga^{9/2}(\xi)
\end{split}
\]
Since $D_2 \Ga^4 - \Ga^4 D_2 : \sD(D_2) \to X$ is the restriction of an element in $\sL(X)$ and since $\| \De S_\la \De \|_\infty \leq 2r (1 + \la)^{-1}$ by \eqref{eq:lemest} we can focus our attention on providing the required estimate for the unbounded operator
\[
\De ( a \Ga^2 T_\la - S_\la \De^2 a) \Ga^4 D_2 : \sD(D_2) \to X
\]
However, by Lemma \ref{l:funalg} we have that
\[
\begin{split}
& \De ( a \Ga^2 T_\la - S_\la \De^2 a) \Ga^4 D_2 \\
& \q = \De S_\la(a \Ga^2 - \De^2 a) T_\la \Ga^4 D_2 \\
& \qqq + \De S_\la \big( \De^{1/2} \rho_\De(1) \De^{1/2} a \Ga + \De^{3/2} \rho_{\De,\Ga}(a) \Ga^{1/2} \big) D_2 T_\la \Ga^4 D_2 \\ 
& \qqq \qq +  \De ( D_1 S_\la )^* \big( \De^{1/2} \rho_{\De,\Ga}(a) \Ga^{3/2} + \De a \Ga^{1/2} \rho_\Ga(1) \Ga^{1/2} \big) T_\la \Ga^4 D_2
\end{split}
\]
as an identity of unbounded operators from $\sD(D_2)$ to $X$. The desired estimate now follows by \eqref{eq:lemest}. 
\end{proof}

We may now gather our work so far into a proof of the following:

\begin{prop}\label{p:compermod}
Suppose that the unbounded modular cycle $(X,D_1,\De)$ is a bounded modular perturbation of the unbounded modular cycle $(X,D_2,\Ga)$. Then it holds that $a \cd (F_{D_1} - F_{D_2}) \in \sK(X)$ for all $a \in A$. In particular, we may conclude that Theorem \ref{t:welbou} holds.
\end{prop}
\begin{proof}
We first remark that it suffices to show that $\De^5 a (F_{D_1} - F_{D_2})\Ga^5 \in \sK(X)$ for all $a \in \C A$. Indeed, this follows since $\C A$ is dense in $A$, since both $(X,F_{D_1})$ and $(X,F_{D_2})$ are bounded Kasparov modules (by Theorem \ref{t:boukas}), and since $\De^5 (\De^5 + 1/n)^{-1} a \to a$ and $a \Ga^5(\Ga^5 + 1/n)^{-1} \to a$ in operator norm for all $a \in A$ (by Definition \ref{d:unbkas}).

Let now $a \in \C A$ be fixed. By Proposition \ref{p:modde1} and Proposition \ref{p:modde2} it is then enough to prove that the difference of unbounded operators
\[
\begin{split}
& \frac{1}{\pi} \int_0^\infty (\la r)^{-1/2} \De^5 a \Ga^2 T_\la \Ga^4 \, d\la \cd D_2  \\
& \qqq - \frac{1}{\pi} \int_0^\infty (\la r)^{-1/2} \De^3 (1 - Y_\la)^{-1} \De^3 R_\la \, d\la \cd D_1 a \Ga^5
: \sD(D_2) \to X
\end{split}
\] 
is the restriction of a compact operator on $X$. However, using Proposition \ref{p:emmcom}, we obtain that
\[
\begin{split}
& \frac{1}{\pi} \int_0^\infty (\la r)^{-1/2} \De^5 a \Ga^2 T_\la \Ga^4 \, d\la \cd D_2(\xi)  \\
& \qqq - \frac{1}{\pi} \int_0^\infty (\la r)^{-1/2} \De^3 (1 - Y_\la)^{-1} \De^3 R_\la \, d\la \cd D_1 a \Ga^5(\xi) \\
& \q = \frac{1}{\pi} \int_0^\infty (\la r)^{-1/2} M_a(\la)(\xi) \, d\la
\end{split}
\]
for all $\xi \in \sD(D_2)$. But this proves the proposition, since the map $M_a : [0,\infty) \to \sK(X)$ is continuous in operator norm and satisfies the estimate $\| M_a(\la) \|_\infty = O\big((1 + \la)^{-3/4}\big)$ as $\la \to \infty$ and as $\la \to 0$. 
\end{proof}

\section{Comparison of products}\label{s:compro}
Let $\C A$ and $\C B$ be operator $*$-algebras satisfying Assumption \ref{a:opealg} and let $C$ be a $\si$-unital $C^*$-algebra. We are now going to compare the interior Kasparov product with the unbounded Kasparov product. The main result of this section states that the unbounded Kasparov product is compatible with the bounded Kasparov product after passing to bounded $KK$-theory by means of the bounded transform. This result follows immediately from \cite[Theorem 10.1]{Kaa:UKM}, Theorem \ref{t:boukas} and Theorem \ref{t:welbou}. We emphasize that \cite[Theorem 10.1]{Kaa:UKM} can not be proved using Kucerovsky's criteria for verifying that an unbounded Kasparov module represents a bounded Kasparov product, see \cite{Kuc:KUM}. Instead, one has to rely directly on the Connes-Skandalis notion of an $F$-connection, see \cite{CoSk:LIF}. The reason for this is that we work with a broader class of unbounded cycles than the rather restrictive class of unbounded Kasparov modules.

\begin{thm}
Suppose that the $C^*$-algebra $A$ is separable. Then the following diagram is commutative
\[
\begin{CD}
M(\C A,\C B) \ti UK_*(\C B,C) @>{\hot_{\C B}}>> UK_*(\C A,C) \\
@V{F \ti F}VV @VV{F}V \\
KK_0(A,B) \ti KK_*(B,C) @>>{\hot_B}> KK_*(A,C)
\end{CD}
\]
where the top row is given by the unbounded Kasparov product and the bottom row is given by the bounded Kasparov product (as constructed in \cite{Kas:OFE}). The columns are given by the bounded transform as described in Definition \ref{d:boutramod} and Definition \ref{d:boutraunb}. 
\end{thm}


\section{Geometric examples of Morita equivalences}\label{s:geomor}
We are now ready to apply our methods to a couple of interesting examples. More precisely, we shall prove Morita equivalence results for hereditary subalgebras, conformally equivalent metrics and free and proper actions of discrete groups. All of our results are well-known in the topological context of $C^*$-algebras and the existence of the corresponding isomorphisms of bounded $KK$-groups therefore follows by the Morita invariance of bounded $KK$-theory, \cite{BrGrRi:SIS,Kas:OFE}. The novelty of this section is, that with a slight amount of extra geometric information it becomes possible to prove $C^1$-versions of the above topological Morita equivalences and hence that the associated operator $*$-algebras admit the same unbounded modular cycles (up to unitary equivalence and bounded modular perturbations). The $C^*$-algebraic versions of the Morita equivalence results presented in this section can be found in \cite{Bro:SIH,BrGrRi:SIS} and \cite[Proposition 1, Chapter 2.7]{Con:NCG}. Remark however that the result on conformally equivalent metrics has no analogue at the $C^*$-algebraic level since the underlying topological space remains the same even though the Riemannian metric is changed.
%

\subsection{Hereditary subalgebras}
Let $\C A$ be an operator $*$-algebra (satisfying Assumption \ref{a:opealg}) and let $\C L \su \C A$ be a closed right ideal. 

We let $L \su A$ denote the closed right ideal obtained as the norm-closure of $\C L$ with respect to the $C^*$-norm on $\C A$.
%
%

\begin{thm}\label{t:hermor}
Suppose that $L$ is countably generated as a Hilbert $C^*$-module over $A$ and that the hereditary $C^*$-subalgebra $L \cap L^* \su A$ is full in the sense that
\[
L^*\cd L := \T{span}_{\cc}\{ \xi^* \cd \eta \mid \xi, \eta \in L \} \su A
\]
is norm-dense in $A$. Then the operator $*$-algebras $\C L \cap \C L^*$ and $\C A$ are Morita equivalent. In particular, we have an explicit even isomorphism
\[
UK_*(\C L \cap \C L^*,B) \cong UK_*(\C A,B) 
\]
of $\zz/2\zz$-graded abelian monoids for any $\si$-unital $C^*$-algebra $B$. 
\end{thm}
\begin{proof}
We give the closed right ideal $\C L \su \C A$ the structure of an operator $*$-correspondence from $\C L \cap \C L^*$ to $\C A$. The bimodule structure is induced by the algebraic operations in $\C A$ and the matrix norms $\| \cd \|_{\C L} : M_m(\C L) \to [0,\infty)$, $m \in \nn$, are given by the restriction of the matrix norms $\| \cd \|_{\C A} : M_m(\C A) \to [0,\infty)$. The inner product is given by $\inn{\xi,\eta}_{\C L} := \xi^* \cd \eta$ for all $\xi, \eta \in \C L$. The $C^*$-completion of $\C L$ then coincides with the closed right ideal $L \su A$ considered as a $C^*$-correspondence from $L \cap L^*$ to $A$. Since $L$ is countably generated over $A$ by assumption and since 
\[
L \cd L^* = \T{span}_{\cc} \big\{ \xi \cd \eta^* \mid \xi, \eta \in L \big\} = L \cap L^*
\]
we obtain that $\C L$ is compact. Notice here that the representation $\pi_{L \cap L^*} : L \cap L^* \to \sL(L)$ induces an isomorphism $L \cap L^* \cong \sK(L)$. We thus have a class $[\C L] \in M(\C L \cap \C L^*, \C A)$.

We now give the left ideal $\C L^* \su \C A$ the structure of an operator $*$-correspondence from $\C A$ to $\C L \cap \C L^*$. The bimodule structure on $\C L^*$ is again induced from the algebraic operations in $\C A$ and the matrix norms are restrictions of the matrix norms on $\C A$. Likewise, we have the inner product $\inn{\xi^*,\eta^*}_{\C L^*} := \xi \cd \eta^*$, $\xi, \eta \in \C L$. The $C^*$-completion of $\C L^*$ coincides with the closed left ideal $L^* \su A$ considered as a $C^*$-correspondence from $A$ to $L \cap L^*$. The assumption that $L^* \cd L$ is norm-dense in $A$ implies that the representation $\pi_A : A \to \sL(L^*)$ induces an isomorphism $A \cong \sK(L^*)$. The fact that $L^*$ is also countably generated follows since $A \cong \sK(L^*)$ is $\si$-unital by Assumption \ref{a:opealg}, see \cite[Proposition 6.7]{Lan:HCM}. We thus have that $\C L^*$ is compact and we get a class $[\C L^*] \in M(\C A, \C L \cap \C L^*)$.

To end the proof, we need to show that $[\C L] \hot_{\C A} [\C L^*] = [\C L \cap \C L^*]$ and that $[\C L^*] \hot_{\C L \cap \C L^*} [\C L] = [\C A]$. This follows by Lemma \ref{l:comboudua} since the product on $\C A$ induces completely bounded bimodule maps $u : \C L \hot_{\C A} \C L^* \to \C L \cap \C L^*$ and $v : \C L^* \hot_{\C L \cap \C L^*} \C L \to \C A$. Moreover, these completely bounded bimodule maps extend to unitary isomorphisms of $C^*$-correspondences $U : L \hot_A L^* \to L \cap L^*$ and $V : L^* \hot_{L \cap L^*} L \to A$.
\end{proof}
%

As a special case of the above theorem we have the following important:

\begin{cor}
The operator $*$-algebras $K_{\C A}$ and $\C A$ are Morita equivalent. In particular, we have an explicit even isomorphism
\[
UK_*(\C A,B) \cong UK_*(K_{\C A},B)  
\]
of $\zz/2\zz$-graded abelian monoids for any $\si$-unital $C^*$-algebra $B$. 
\end{cor}
\begin{proof}
We notice that the row correspondence over $\C A$ is a closed right ideal inside the compacts over $\C A$, $\ell^2(\nn,\C A)^t \su K_{\C A}$. The closed left ideal $\big(\ell^2(\nn,\C A)^t\big)^* \su K_{\C A}$ then agrees with the column correspondence over $\C A$, $\big(\ell^2(\nn,\C A)^t\big)^* = \ell^2(\nn,\C A)$ and we conclude that $\ell^2(\nn,\C A)^t \cap \big( \ell^2(\nn,\C A)^t\big)^*  = \C A$. The result of the present corollary now follows by an application of Theorem \ref{t:hermor}.
\end{proof}
%

\subsection{Conformal equivalence}\label{ss:conf}
Throughout this subsection we will fix a smooth Riemannian manifold $\C M$ with smooth Riemannian metric $g : \Ga^\infty(T^*_\rr \C M) \ti \Ga^\infty(T^*_\rr \C M) \to C^\infty(\C M,\rr)$ (where $T^*_\rr \C M \to \C M$ denotes the real cotangent bundle over $\C M$). Furthermore, we will consider a smooth and strictly positive function $\al : \C M \to (0,\infty)$. We then have the conformally equivalent metric 
\[
g_\al : \Ga^\infty(T^*_\rr \C M) \ti \Ga^\infty(T^*_\rr \C M) \to C^\infty(\C M,\rr) \q  g_\al(\om, \si)(x) := \al^2(x) \cd g(\om,\si)(x) 
\]
We will also use the notation $g, g_\al : \Ga^\infty(T^* \C M) \ti \Ga^\infty(T^* \C M) \to C^\infty(\C M)$ for the associated pairings of smooth sections of the complexified cotangent bundle $T^* \C M \to \C M$. We impose no completeness conditions on the Riemannian manifolds involved.

Let $d : C^\infty(\C M) \to \Ga^\infty(T^* \C M)$ denote the exterior derivative.

The Riemannian metrics give rise to two (unitarily isomorphic) $C^*$-correspondences $\Ga_0(T^* \C M)$ and $\Ga_{0,\al}(T^* \C M)$ both of them from $C_0(\C M)$ to $C_0(\C M)$. They are obtained by completing the continuous compactly supported sections of the cotangent bundle $T^* \C M \to \C M$ with respect to the two different (pre)-inner products:
\[
g \, , \, \, g_\al : \Ga_c(T^* \C M) \ti \Ga_c(T^* \C M) \to C_0(\C M)
\]
Associated to these two Hilbert $C^*$-modules we obtain two algebra homomorphisms:
\[
\begin{split}
& \pi : C_c^\infty(\C M) \to \sL\big( C_0(\C M) \op \Ga_0(T^*\C M) \big) \q \T{and} \\
& \pi_{\al} : C_c^\infty(\C M) \to \sL\big( C_0(\C M) \op \Ga_{0,\al}(T^*\C M) \big)
\end{split}
\]
where $\pi$ is given by
\[
\pi(f)\ma{c}{h \\ \om} := \ma{c}{f \cd h \\ f \cd \om + df \cd h} \q f \in C^\infty_c(\C M) \, , \, \, h \in C_0(\C M) \, , \, \, \om \in \Ga_0(T^* \C M) 
\]
and a similar formula defines $\pi_\al$. We remark that even though the $C^*$-correspondences $C_0(\C M) \op \Ga_0(T^* \C M)$ and $C_0(\C M) \op \Ga_{0,\al}(T^* \C M)$ are unitarily isomorphic it might not be possible to find a unitary isomorphism which intertwines the two algebra homomorphisms $\pi$ and $\pi_\al$.

We define the matrix norms $\| \cd \|_1 , \| \cd \|_{1,\al} : M_m(C_c^\infty(\C M)) \to [0,\infty)$, $m \in \nn$, by the formulae
\[
\begin{split}
& \| f \|_1 := \max\big\{ \| \pi(f) \|_\infty , \,  \| \pi(f^*) \|_\infty \big\}
\q \T{and} \\
& \| f \|_{1,\al} := \max\big\{ \| \pi_\al(f) \|_\infty , \| \pi_\al(f^*) \|_\infty \big\}
\end{split}
\]
where the notation ``$\|\cd \|_\infty$'' is applied both for the operator norm on $\sL\big( C_0(\C M) \op \Ga_0(T^*\C M) \big)$ and on $\sL\big( C_0(\C M) \op \Ga_{0,\al}(T^*\C M) \big)$.
%

We then define the operator $*$-algebras $\C A$ and $\C A_{\al}$ as the completions of the $*$-algebra $C_c^\infty(\C M)$ with respect to the norms $\| \cd \|_1 \T{ and } \|\cd \|_{1,\al} : C_c^\infty(\C M) \to [0,\infty)$, respectively. We notice that both of these operator $*$-algebras consist of continuously differentiable functions on $\C M$ vanishing at infinity. The difference is that the exterior derivative vanishes at infinity with respect to two different hermitian structures, namely $g \T{ and } g_\al : \Ga^\infty(T^* \C M) \ti \Ga^\infty(T^* \C M) \to C^\infty(\C M)$. For more details on the operator $*$-algebras $\C A$ (and $\C A_\al$) we refer to \cite[Section 2.3]{KaLe:SFU} and \cite[Proposition 3.4]{Kaa:SBB}. 

The operator $*$-algebras $\C A$ and $\C A_{\al}$ may be equipped with the supremum norm and they then satisfy Assumption \ref{a:opealg}. In both cases, the $C^*$-completion agrees with the $C^*$-algebra of continuous functions on $\C M$ that vanishes at infinity, $C_0(\C M)$.


\begin{thm}
Suppose that $(1 + \al)^{-2} \cd g(d \al,d\al) : \C M \to (0,\infty)$ is a bounded function. Then the operator $*$-algebras $\C A$ and $\C A_\al$ are Morita equivalent. In particular, we have an explicit even isomorphisms
\[
UK_*(\C A,B) \cong UK_*(\C A_\al, B) 
\]
of $\zz/2\zz$-graded abelian monoids for any $\si$-unital $C^*$-algebra $B$. 
\end{thm}
\begin{proof}
We start by arguing that we may assume that $\al$ and $g(d\al,d\al) : \C M \to (0,\infty)$ are bounded. Indeed, since Morita equivalence is an equivalence relation, it suffices to show that $\C A_{1 + \al}$ and $\C A$ are Morita equivalent and that $\C A_{1 + \al}$ and $\C A_{\al}$ are Morita equivalent. But we clearly have that $(1 + \al)^{-1}$ and $\al(1 + \al)^{-1}$ are bounded and moreover that
\[
\begin{split}
& (1 + \al)^2 \cd g\big( d(\al(1 + \al)^{-1}), d(\al(1 + \al^{-1}))\big) \\
& \q = (1 + \al)^2 \cd g\big( d((1 + \al)^{-1}),d((1+\al)^{-1})\big)
= (1 + \al)^{-2} \cd g(d\al,d\al)
\end{split}
\]
is bounded by assumption.

Thus, assume (w.l.o.g.) that $\al$ and $g(d\al,d\al) : \C M \to (0,\infty)$ are bounded. 

It then holds that the identity map $i : C_c^\infty(\C M) \to C_c^\infty(\C M)$ induces a completely bounded $*$-homomorphism $i : \C A \to \C A_\al$. Furthermore, the multiplication by $\al$, $m(\al) : C_c^\infty(\C M) \to C_c^\infty(\C M)$, induces a completely bounded map $m(\al) : \C A_\al \to \C A$.

We equip $\C A_\al$ with the structure of an operator $*$-correspondence $\C X$ from $\C A$ to $\C A_\al$. The right module structure is given by the algebra operations on $\C A_\al$ and the left action of $\C A$ is defined by $a \cd \xi := i(a) \cd \xi$ for all $a \in \C A$, $\xi \in \C A_\al$ (where $\cd$ denotes the product in $\C A_\al$). The inner product on $\C X$ is defined by $\inn{\xi,\eta} := \xi^* \cd \eta$, $\xi, \eta \in \C A_\al$. The matrix norms on $\C X$ agree with the matrix norms $\| \cd \|_{1,\al} : M_m(\C A_\al) \to [0,\infty)$ on $\C A_\al$. The $C^*$-completion of $\C X$ is then nothing but the $C^*$-algebra $C_0(\C M)$ of continuous functions vanishing at infinity when considered as a $C^*$-correspondence from $C_0(\C M)$ to $C_0(\C M)$. It therefore follows that $\C X$ is compact as an operator $*$-correspondence from $\C A$ to $\C A_{\al}$ and we thus have a class $[\C X] \in M(\C A, \C A_\al)$.

We now equip $\C A_\al$ with the structure of an operator $*$-correspondence $\C Y$ from $\C A_\al$ to $\C A$. The left module structure is given by the algebra operations on $\C A_\al$ and the right action of $\C A$ is given by $\xi \cd a := \xi \cd i(a)$ for all $\xi \in \C A_\al$, $a \in \C A$. The matrix norms on $\C Y$ agrees with the matrix norms $\| \cd \|_{1,\al} : M_m(\C A_\al) \to [0,\infty)$ on $\C A_\al$. To define the inner product on $\C Y$ we consider the pairing
\[
\inn{\cd,\cd}_\al : C_c^\infty(\C M) \ti C_c^\infty(\C M) \to C_c^\infty(\C M) \su \C A \q \inn{f,h}_\al := \frac{\al^2}{\|m(\al)\|_{\T{cb}}^2} \cd f^* \cd h
\]
where $\| m(\al) \|_{\T{cb}}$ is the completely bounded norm of the completely bounded map $m(\al) : \C A_\al \to \C A$, see Definition \ref{d:opspa}. We remark that the Cauchy-Schwarz inequality is satisfied since
\[
\begin{split}
\| \inn{f,h}_\al \|_1 
& = \frac{1}{\| m(\al) \|_{\T{cb}}^2 } \cd \| (\al \cd f)^* \cd (\al \cd h) \|_1 \\
& \leq \frac{1}{\| m(\al) \|_{\T{cb}}^2 } \cd \| m(\al)(f) \|_1 \cd \| m(\al)(h) \|_1
\leq \| f \|_{1,\al} \cd \| h \|_{1,\al}
\end{split}
\]
for all $f, h \in M_m(C_c^\infty(\C M))$, $m \in \nn$. The $C^*$-completion $Y$ is unitarily isomorphic to the $C^*$-algebra $C_0(\C M)$ considered as a $C^*$-correspondence from $C_0(\C M)$ to $C_0(\C M)$. The unitary isomorphism is induced by the multiplication operator $m(\al)/\| m(\al) \|_{\T{cb}} : C_c^\infty(\C M) \to C_c^\infty(\C M)$. It thus follows that the operator $*$-correspondence $\C Y$ is compact and we obtain a class $[\C Y]$ in $M(\C A_\al, \C A)$.

To finish the proof of the theorem we need to show that the identities $[\C Y \hot_{\C A} \C X] = [\C A_\al]$ and that $[\C X \hot_{\C A_\al} \C Y]  = [\C A]$ hold in $M(\C A_\al, \C A_\al)$ and $M(\C A,\C A)$, respectively. To this end we apply Lemma \ref{l:comboudua} twice. Indeed, for the first of the identities we note that the composition
\[
\begin{CD}
\C A_\al \wot \C A_\al @>{m}>> \C A_\al @>{m(\al)/\| m(\al) \|_{\T{cb}}}>> \C A @>{i}>> \C A_\al
\end{CD}
\]
is completely bounded (where $m$ is given by multiplication in $\C A_\al$) and that this map induces a completely bounded bimodule map $u : \C Y \hot_{\C A} \C X \to \C A_\al$. Moreover, it holds that $u$ extends to a unitary isomorphism $U : Y \hot_{C_0(\C M)} C_0(\C M) \to C_0(\C M)$. For the second identity we note that the composition
\[
\begin{CD}
\C A_\al \wot \C A_\al @>{m}>> \C A_\al @>{m(\al)/\| m(\al) \|_{\T{cb}}}>> \C A
\end{CD}
\]
is completely bounded and that this composition induces a completely bounded bimodule map $v : \C X \hot_{\C A_\al} \C Y \to \C A$. Furthermore, this bimodule map extends to a unitary isomorphism $V : C_0(\C M) \hot_{C_0(\C M)} Y \to C_0(\C M)$. This ends the proof of the theorem.
\end{proof}

\subsection{Crossed products by discrete groups}
Throughout this subsection we let $\C M$ be a Riemannian manifold and we let $G$ be a countable discrete group. We will assume that we have a right-action $\C M \ti G \to \C M$ by diffeomorphisms of $\C M$. The diffeomorphism associated to an element $g \in G$ will be denoted by $\phi_g : \C M \to \C M$, $\phi_g(x) = x \cd g$. 

For each $g \in G$ we let $(d \phi_g)(x) : T_x \C M \to T_{\phi_g(x)} \C M$ denote the derivative of $\phi_g : \C M \to \C M$ evaluated at the point $x \in \C M$.

The following conditions will be relevant:

\begin{assu}\label{a:grpact}
It will be assumed that
\begin{enumerate}
\item The action of $G$ on $\C M$ is free and proper;
\item For each $g \in G$ we have that
\[
\| d \phi_g \|_\infty := \sup_{x \in \C M} \| (d \phi_g)(x) \|_x < \infty
\]
where $\| \cd \|_x : \sL( T_x \C M, T_{\phi_g(x)} \C M) \to [0,\infty)$ refers to the operator norm;
\item We have that
\[
\sup_{g \in G} \| d \phi_g \|_\infty < \infty
\]
\end{enumerate}
\end{assu}

For each $g \in G$ we define the $*$-isomorphism $\al_g : C_c^\infty(\C M) \to C_c^\infty(\C M)$ by $\al_g : f \mapsto f \ci \phi_g$. Notice that $\al_g \ci \al_h = \al_{gh}$ for all $g,h \in G$.

\subsubsection{The crossed product operator $*$-algebra}
Let us assume that $(2)$ and $(3)$ of Assumption \ref{a:grpact} are satisfied.

We consider the $*$-algebra $C_c(G, C_c^\infty(\C M))$ consisting of finite sums $\sum_{g \in G} f_g U_g$ of elements in $C_c^\infty(\C M)$ indexed by the countable discrete group $G$. We recall that the sum is defined pointwise and that the product is the convolution product: $f U_g \cd h U_k := f \al_g( h ) U_{gk}$. The involution is defined by $(f U_g)^* := \al_g^{-1}(f^*) U_{g^{-1}}$.

We let $\ell^2\big( G, C_0(\C M) \op \Ga_0(T^* \C M) \big)$ denote the Hilbert $C^*$-module over $C_0(\C M)$ consisting of square summable sequences indexed by $G$ with values in the Hilbert $C^*$-module $C_0(\C M) \op \Ga_0(T^* \C M)$. We recall that the $C_0(\C M)$-valued inner product on $\Ga_0(T^* \C M)$ comes from the Riemannian metric on $\C M$. Equivalently, the Hilbert $C^*$-module $\ell^2\big( G, C_0(\C M) \op \Ga_0(T^* \C M) \big)$ can be described as the exterior tensor product $\ell^2(G) \hot \big( C_0(\C M) \op \Ga_0(T^* \C M) \big)$ of the Hilbert space $\ell^2(G)$ and the Hilbert $C^*$-module $C_0(\C M) \op \Ga_0(T^* \C M)$.

We define the algebra homomorphism
\[
\pi : C_c(G, C_c^\infty(\C M)) \to \sL\big( \ell^2\big( G, C_0(\C M) \op \Ga_0(T^* \C M) \big) \big)
\]
by the formula:
\[
\pi(f U_g) \ma{c}{h \\ \om} \de_k := \ma{c}{ \al_{gk}^{-1}(f) \cd h \\ d\big( \al_{gk}^{-1}(f)\big) \cd h + \al_{gk}^{-1}(f) \cd \om} \de_{gk}
\]
The fact that $\pi(f U_g)$ defines a bounded adjointable operator relies on $(2)$ and $(3)$ of Assumption \ref{a:grpact}.

We define the matrix norms $\| \cd \|_1 : M_m\big( C_c\big( G, C_c^\infty(\C M) \big)\big) \to [0,\infty)$, $m \in \nn$, by the formula:
\[
\| f \|_1 := \T{max}\big\{ \| \pi(f) \|_\infty \, , \, \, \| \pi(f^*) \|_\infty \big\}
\]
We let $C^1_0(\C M) \rti_r G$ denote the operator $*$-algebra obtained as the completion of $C_c\big( G, C_c^\infty(\C M) \big)$ with respect to the norm $\| \cd \|_1$.

The operator $*$-algebra $C^1_0(\C M) \rti_r G$ has a $C^*$-norm defined by 
\[
\| f \|_\infty := \big\| \ma{cc}{1 & 0 \\ 0 & 0} \pi(f) \ma{cc}{1 & 0 \\ 0 & 0} \big\|_\infty \q f \in C^1_0(\C M) \rti_r G
\]
and the associated $C^*$-completion agrees with the reduced crossed product $C^*$-algebra $C_0(\C M) \rti_r G$. Remark here that the matrix $\ma{cc}{1 & 0 \\ 0 & 0}$ projects onto the first summand in $\ell^2\big(G, C_0(\C M) \op \Ga_0(T^* \C M) \big)$. It follows that the operator $*$-algebra $C^1_0(\C M) \rti_r G$ satisfies the conditions of Assumption \ref{a:opealg}.

%
%
%

\subsubsection{The operator $*$-correspondence}
Let us assume that all of the conditions in Assumption \ref{a:grpact} are satisfied.
\medskip

We let $C_b^1(\C M)$ denote the unital $*$-algebra consisting of all bounded continuously differentiable functions $f : \C M \to \cc$ with bounded exterior derivative $df \in \Ga(T^* \C M)$. We remark that $C_b^1(\C M)$ becomes an operator $*$-algebra when equipped with the matrix norms
\[
\| \cd \|_1 : M_m\big( C_b^1(\C M) \big) \to [0,\infty) \q \| f \|_1 := \T{max}\big\{ \|\pi(f)\|_\infty \, , \, \, \| \pi(f^*) \|_\infty \big\}
\]
where $\pi : C_b^!(\C M) \to \sL\big( C_0(\C M) \op \Ga_0(T^* \C M) \big)$ denotes the algebra homomorphism defined by $\pi(f) := \ma{cc}{f & 0 \\ df & f}$ (see Subsection \ref{ss:conf} for more details).
%

We let $C_c^1(\C M/G) \su C_b^1(\C M)$ denote the $*$-subalgebra of $G$-invariant compactly supported functions. Thus, $C_c^1(\C M/G)$ consists of all functions $f \in C^1_b(\C M)$ such that $\al_g(f) = f$ for all $g \in G$ and such that the induced map $[f] : \C M/G \to \cc$ has compact support. The completion $C_0^1(\C M/G) \su C_b^1(\C M)$ is then an operator $*$-algebra. We may equip $C^1_0(\C M/G)$ with the $C^*$-norm given by the supremum norm on $\C M/G$ and the operator $*$-algebra $C^1_0(\C M/G)$ then satisfies Assumption \ref{a:opealg}. The $C^*$-completion of $C^1_0(\C M/G)$ agrees with the $C^*$-algebra $C_0(\C M/G)$ of continuous functions vanishing at infinity on the quotient space $\C M/G$.
\medskip

We equip the vector space $C_c^1(\C M)$ with the structure of a $C_c(G, C_c^1(\C M))$-$C_c^1(\C M/G)$-bimodule with left action defined by
\[
(f U_g ) \cd \xi := f \cd \al_g(\xi)
\]
and with right action induced by the algebra structure on $C_b^1(\C M)$. Furthermore, we define the pairing
\begin{equation}\label{eq:rinn}
\inn{\cd,\cd}_R : C_c^1(\C M) \ti C_c^1(\C M) \to C_c^1(\C M/G) \q \inn{\xi,\eta}_R := \sum_{g \in G} \al_g^{-1}(\xi^* \cd \eta) 
\end{equation}

In order to construct an operator $*$-correspondence $\C X$ from $C^1_0(\C M) \rti_r G$ to $C_0^1(\C M/G)$ out of the bimodule $C_c^1(\C M)$ we represent all of our data as bounded adjointable operators between Hilbert $C^*$-modules.

We define two linear maps
\[
\begin{split}
& \pi : C_c^1(\C M) \to \sL\Big( C_0(\C M) \op \Ga_0(T^* \C M), \ell^2\big(G, C_0(\C M) \op \Ga_0(T^* \C M)\big) \Big) \q \T{and} \\
& \pi^t : C_c^1(\C M) \to \sL\Big( \ell^2\big(G, C_0(\C M) \op \Ga_0(T^* \C M)\big) , C_0(\C M) \op \Ga_0(T^* \C M) \Big)
\end{split}
\]
by the formulae:
\[
\begin{split}
& \pi(\xi)\ma{c}{h \\ \om} := \sum_g \ma{c}{ \al_g^{-1}(\xi) \cd h \\ d\big( \al_g^{-1}(\xi)\big) \cd h + \al_g^{-1}(\xi) \cd \om} \de_g \q \T{and} \\
& \pi^t(\xi)\big( \ma{c}{h \\ \om }\de_g \big) := \ma{c}{ \al_g^{-1}(\xi) \cd h \\ d\big( \al_g^{-1}(\xi)\big) \cd h + \al_g^{-1}(\xi) \cd \om}
\end{split}
\]
The fact that these formulae define bounded adjointable operators relies on $(1)$, $(2)$ and $(3)$ of Assumption \ref{a:grpact}.

It can then be verified that our different representations satisfy the following compatibility relations:
\begin{enumerate}
\item $\pi( f U_g) \cd \pi(\xi) = \pi\big( (f U_g)\cd \xi \big)$; 
\item $\pi(\xi) \cd \pi( h ) = \pi( \xi \cd h)$;
\item $\pi^t(\xi^*) \cd \pi\big( (f U_g)^*\big) = \pi^t\big( ((f U_g) \cd \xi)^*\big)$;
\item $\pi(h^*) \cd \pi^t(\xi^*) = \pi^t\big( (\xi \cd h)^*\big)$
\end{enumerate}
for all $f U_g \in C_c(G, C_c^1(\C M))$, $\xi \in C_c^1(\C M)$ and $h \in C_c^1(\C M/G)$. Moreover, it holds that
\[
\pi^t(\eta^*) \cd \pi(\xi) = \pi( \inn{\eta,\xi}_R)
\]
for all $\xi,\eta \in C_c^1(\C M)$. We may thus define matrix norms on $C_c^1(\C M)$, $\| \cd \|_1 : M_m(C_c^1(\C M)) \to [0,\infty)$, $m \in \nn$, by 
\[
\| \xi \|_1 := \T{max}\big\{ \| \pi(\xi) \|_\infty \, , \, \, \| \pi^t(\xi^*) \|_\infty \big\}
\]

We let $\C X$ denote the operator $*$-correspondence from $C_0^1(\C M) \rti_r G$ to $C_0^1(\C M/G)$ obtained as the completion of $C_c^1(\C M)$ in the norm $\| \cd \|_1$. The $C^1_0(\C M/G)$-valued inner product is induced by the pairing in \eqref{eq:rinn} and the bimodule structure is induced by the bimodule structure on $C^1_c(\C M)$. This operator $*$-correspondence satisfies Assumption \ref{a:opecor} and we let $X$ denote the $C^*$-correspondence from $C_0(\C M) \rti_r G$ to $C_0(\C M/G)$ obtained as the $C^*$-completion of $\C X$.
\medskip

Let us now consider the formal dual operator bimodule $\C X^*$ from $C_0^1(\C M/G)$ to $C_0^1(\C M) \rti_r G$ (see Section \ref{s:opestar}). We are going to equip this operator bimodule with the structure of an operator $*$-correspondence from $C_0^1(\C M/G)$ to $C_0^1(\C M) \rti_r G$. We first define the $C_c(G,C_c^1(\C M))$-valued inner product on the formal dual bimodule $C_c^1(\C M)^*$:
\[
\inn{\cd,\cd}_L : C_c^1(\C M)^* \ti C_c^1(\C M)^* \to C_c(G, C_c^1(\C M)) \q \inn{\xi^*,\eta^*}_L := \sum_{g \in G} \xi \cd \al_g(\eta^*) U_g
\]
We then remark that this pairing is compatible with our representations since
\[
\pi(\xi) \cd \pi^t(\eta^*) = \pi( \inn{\xi^*,\eta^*}_L)
\]
for all $\xi, \eta \in C_c^1(\C M)$. The pairing $\inn{\cd,\cd}_L$ therefore extends to a $C_0^1(\C M) \rti_r G$-valued inner product on $\C X^*$. The operator $*$-correspondence obtained in this way satisfies Assumption \ref{a:opecor} and the $C^*$-completion agrees with the dual $C^*$-correspondence $X^* = \sK(X,C_0(\C M/G))$.


\subsubsection{Morita equivalence}

\begin{thm}
Suppose that the conditions in Assumption \ref{a:grpact} are satisfied. Then the operator $*$-algebras $C^1_0(\C M) \rti_r G$ and $C^1_0(\C M/G)$ are Morita equivalent. In particular we have an explicit even isomorphism of abelian monoids
\[
UK_*(C^1_0(\C M) \rti_r G, B) \cong UK_*(C^1_0(\C M/G), B) 
\]
for any $\si$-unital $C^*$-algebra $B$. 
\end{thm}
\begin{proof}
The operator $*$-correspondences $\C X$ and $\C X^*$ are both compact and we thus obtain classes $[\C X] \in M( C_0^1(\C M) \rti_r G, C_0^1(\C M/G))$ and $[\C X^*] \in M( C_0^1(\C M/G), C_0^1(\C M) \rti_r G)$. It therefore suffices to check that $\C X \hot_{C_0^1(\C M/G)} \C X^* \sim_d C_0^1(\C M) \rti_r G$ and $\C X^* \hot_{C_0^1(\C M) \rti_r G} \C X \sim_d C_0^1(\C M/G)$. To this end, we remark that the inner product $\inn{\cd,\cd}_L$ on $\C X^*$ induces a completely bounded bimodule map
\[
u : \C X \hot_{C_0^1(\C M/G)} \C X^* \to C_0^1(\C M) \rti_r G \q \xi \ot \eta^* \mapsto \inn{\xi^*,\eta^*}_L
\]
and moreover that the inner product $\inn{\cd,\cd}_R$ on $\C X$ induces a completely bounded bimodule map
\[
v : \C X^* \hot_{C^1_0(\C M) \rti_r G} \C X \to C_0^1(\C M/G) \q \xi^* \ot \eta \mapsto \inn{\xi,\eta}_R
\]
Since these two completely bounded maps induce unitary isomorphisms $U : X \hot_{C_0(\C M/G)} X^* \to C_0(\C M) \rti_r G$ and $V : X^* \hot_{C_0(\C M) \rti_r G} X \to C_0(\C M/G)$ we have proved the present theorem by an application of Lemma \ref{l:comboudua}.
\end{proof}


\bibliographystyle{amsalpha-lmp}

\providecommand{\bysame}{\leavevmode\hbox to3em{\hrulefill}\thinspace}
\providecommand{\MR}{\relax\ifhmode\unskip\space\fi MR }
\providecommand{\MRhref}[2]{%
  \href{http://www.ams.org/mathscinet-getitem?mr=#1}{#2}
}
\providecommand{\href}[2]{#2}

\end{document}